\documentclass[10pt]{article}

\usepackage{graphicx}
\usepackage{tikz}
\usepackage{atbegshi}

\AtBeginShipoutFirst{%
    \begin{tikzpicture}[remember picture, overlay]
        \node[anchor=north west, xshift=1in, yshift=-1.5cm] at (current page.north west) {%
            \includegraphics[width=4cm]{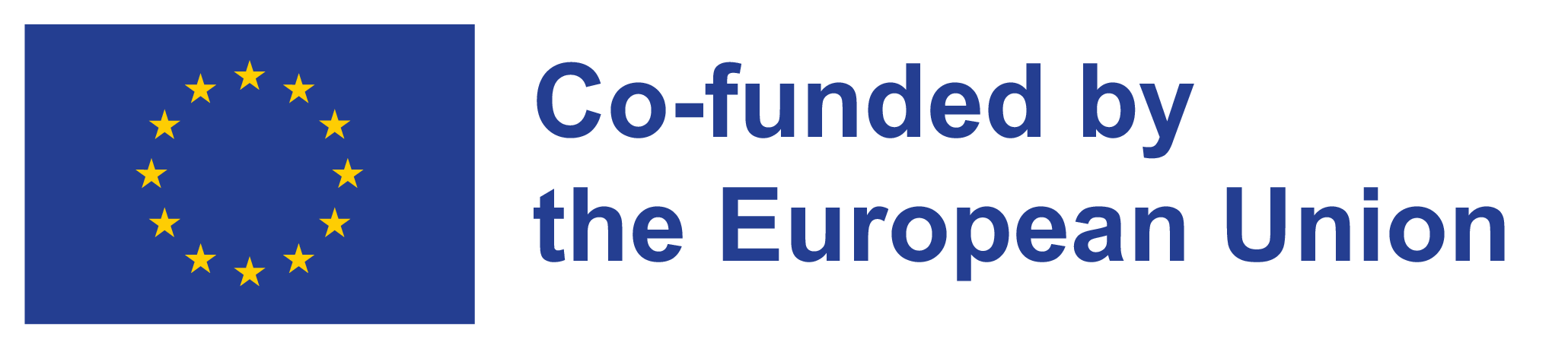} 
        };
    \end{tikzpicture}%
}

\usepackage[ngerman,french,english]{babel}

\usepackage{bbm}

\usepackage{pdfsync}

\usepackage{amsmath}                                                        
\numberwithin{equation}{section}
\usepackage{graphicx}                                                       
\usepackage{subfigure}
\usepackage{enumitem}                                                    
\usepackage{hyperref}
\usepackage{amssymb}                                                        
\usepackage[mathscr]{eucal}                                             
\usepackage{cancel}                                                             
\usepackage[normalem]{ulem}                                                                 
\usepackage{pstricks}
\usepackage{rotating}
\usepackage{lscape}
\usepackage[paperwidth=8.5in,paperheight=11in,top=1.00in, bottom=1.00in, left=1.00in, right=1.00in]{geometry}
\usepackage{mathtools}                                                      
\mathtoolsset{showonlyrefs=true}                                    
\MakeRobust{\eqref}

\usepackage{dsfont}
\usepackage{comment}
\usepackage{wasysym}

\usepackage{mathdots}
\usepackage{amsthm}                                                             
\allowdisplaybreaks

\theoremstyle{plain}
\newtheorem{theorem}{Theorem}
\numberwithin{theorem}{section}

\newtheorem{lemma}[theorem]{Lemma}                              

\theoremstyle{definition}
\newtheorem{definition}[theorem]{Definition}

\newtheorem{notation}[theorem]{Notation}
\newtheorem{remark}[theorem]{Remark}

\counterwithout{figure}{section}

\def \b {{\beta}}

\def \R {\mathbb{R}}
\def \p {\partial}
\def \t {\tau}

\newcommand\Eb{\mathbb{E}}

\newcommand\Pb{\mathbb{P}}

\newcommand\Wb{\mathbb{W}}
\newcommand\Ib{\mathbb{I}}

\newcommand\eps{\varepsilon}

\newcommand{\NCL}[1]{\mathbb{D} ([0,{#1}];D)}

\newcommand\contr{{\bf I}}

\renewcommand{\red}[1]{\textcolor{red}{#1}}

\def \R  {{\mathbb {R}}}

\def \eps {{\varepsilon}}
\def \t {{\tau}}

\def \p {{\partial}}

\def \b {{\beta}}

\def \It\^o {It\^o }

\def \R {{\mathbb {R}}}

\def \eps {{\varepsilon}}

\def \t {{\tau}}
\def \t {{\tau}}

\def \tilde {\widetilde}

\def \Ã  {{\`a }}
\def \Ã¨ {{\`e }}
\def \Ã² {{\`o }}
\def \Ã¹ {{\`u }}

\def\cI{{\mathcal I}}

\def\cL{{\mathcal L}}

\def\cR{{\mathcal R}}
\def\cS{{\mathcal S}}

\def\1{{\mathbf{1}}}

\DeclareMathOperator*{\argmax}{arg\,max}

\begin{document}

\title{Nash equilibrium in a singular stochastic game between two renewable power producers with price impact}

\author{Stefano Pagliarani
\thanks{Dipartimento di Matematica, Universit\`a di Bologna, Bologna, Italy
  \textbf{E-mail}: stefano.pagliarani9@unibo.it.} 
  \and Antonello Pesce
\thanks{Dipartimento di Matematica, Universit\`a di Bologna, Bologna, Italy
  \textbf{E-mail}: antonello.pesce2@unibo.it} 
  \and Tiziano Vargiolu
\thanks{Dipartimento di Matematica ``Tullio Levi-Civita”, Universit\`a di Padova, Padova, Italy
\textbf{E-mail}: vargiolu@math.unipd.it}}

\date{}

\maketitle

\begin{abstract}
We study the singular stochastic game, formulated in \cite{MR4352705}, 
between two agents aiming at maximizing their profits by installing photovoltaic panels and selling the produced electricity, net of installation costs, in the case that their cumulative installations have an impact on power prices. We first solve explicitly the static, one-step, version of the game, and find that Nash equilibria divide the state-space into four regions: one where both players are idle, two where only one player installs new panels, and one where both players install. In some particular regimes, we find that the latter may not be uniquely distinguished from the previous two. 
We then consider the dynamic, continuous-time, problem. Led by the intuition garnered in the static case, we assume a free-boundary structure similar to that arising in the one-step game and provide a rigorous verification theorem for the corresponding system of free-boundary HJB equations, also taking into account the lack of smoothness of the value functions near the free boundaries. 
Finally, for each regular solution of the HJB system, we show that there exists a unique equilibrium strategy, which is obtained as the solution to the Skorokhod-type problem associated with the free boundary.
\end{abstract}

\noindent

\noindent

{\noindent\textbf{Acknowledgements}: The authors have received funding from the European Union’s Horizon Europe research and innovation programme under the Marie Sklodowska-Curie Actions Staff Exchanges (Grant Agreement No. 101183168, Call: HORIZON-MSCA-2023-SE-01). The research of AP was supported by the Italian Ministry of University and Research (MUR) under the National Operational Programme “Research and Innovation” 2014–2020, Action IV.6 “Research contracts on green topics”, financed through ESF REACT-EU resources (CUP: J41B21012140007).  
The research of SP was partially supported by the PRIN22 project no. CUP\_E53C23001670001. 
{The research of TV is supported by the INdAM - GNAMPA Project code CUP E53C23001670001 and by the projects funded by the European Union - Next Generation EU, Mission 4 Component 1, 2022BEMMLZ ``Stochastic control and games and the role of information'' CUP C53D23002430006, P20224TM7Z "Probabilistic methods for energy transition"  CUP C53D23008390001. 
The authors wish to thank for fruitful discussions Salvatore Federico.
}

\vspace{5pt}

{\noindent\textbf{Disclaimer}: Funded by the European Union. Views and opinions expressed are however those of the author(s) only and do not necessarily reflect those of the European Union or the European Education and Culture Executive Agency (EACEA). Neither the European Union nor EACEA can be held responsible for them.}

\tiny
\tableofcontents
\normalsize

\section{Introduction}

Amid the ongoing green energy transition, the problem of determining the optimal installation of renewable energy sources (RES) has attracted increasing interest from power producers and regulators. These problems typically lead to multi-player optimization problems (or games) under uncertainty, which may display various mathematical features depending on the structure of the underlying market conditions, such as the number and size of the agents involved, the type of cost–reward functional of each agent, or the impact of renewable production on energy prices, among others. One common feature, however, shared by most RES optimization problems is the irreversibility of investments in new power production, which, from a mathematical standpoint, naturally translates into singular control problems with monotone controls.

In \cite{MR4305783}, the authors studied the problem of quantifying the optimal, irreversible, installation of photovoltaic panels, over an infinite time horizon, by a single agent seeking to maximize the expected profit from the sale of the energy produced, net of installation costs. Critically, the firm is assumed to be sufficiently large for its investment decisions to affect the electricity price. This effect, hereafter referred to as \emph{price impact}, results in a singular stochastic control problem where the control, namely the cumulative installation, affects both the dynamics of the state variable, namely the electricity spot price, and the cost-reward functional. Following mean-reverting specifications that are classical in energy markets (see e.g. \cite{Cartea2005}, \cite{Benth2008} and the references therein), the electricity  price is assumed to follow a controlled Ornstein-Uhlenbeck (OU) process, with the \textit{price impact} implemented so that the long-term mean decreases linearly with the cumulative installation.  
Furthermore, the cumulative installation is constrained by a finite maximum capacity $\theta>0$. 
This additional \emph{finite-fuel} feature forces one to consider the installed power as an additional fully controlled, degenerate, state variable. 
The resulting optimal strategy divides the state space into a waiting region, in which the control is idle, and an activation region, in which the agent installs the minimal capacity necessary to keep the power price lower than a nonlinear function of the cumulative installed capacity. This function determines the free boundary separating the two regions ,and is characterized as the solution of an ordinary differential equation. 
Barrier/trigger-type investment rules are well known in the literature on irreversible investment under uncertainty (real options), where the value function typically solves a free-boundary problem and the state-space splits into a \textit{waiting} and \textit{investing} region (see the classic texts \cite{Dixit1994, MR2898980}). In the continuous and cumulative investment setting, this structure corresponds to a finite-fuel monotone singular control problem, characterized by a variational inequality with a gradient constraint and an associated free boundary (see \cite{MR698818, MR762624, MR653144, MR2179357}.)

{In this paper, we continue 
the analysis initiated in \cite{MR4352705}, where the authors took a first step toward extending the previous problem to a market with $N$ companies, where price impact is induced by their aggregate installations, subject to a global capacity constraint $\theta$. The interactions between the firms can be of different nature and several equilibrium notions can be considered accordingly; among these, the two most well-studied are those of Pareto optimality and Nash equilibrium.
The problem of finding a Pareto optimum was fully solved in \cite{MR4352705}. On the other hand, the problem of finding a Nash equilibrium proves to be much more involved and was solved, in \cite{MR4352705}, only with no price impact. 
Here we aim to address the characterization of Nash equilibria in the general case, namely with non-zero price impact. In accordance with the single-player setting, the problem is formulated as a nonzero-sum game with finite-fuel singular control. In particular, Player $i$ controls their own cumulative installation process, $Y^i_t$, and the joint evolution of these controlled variables affects the electricity price, $X_t$. To promote readability, and to convey a neat intuition of the structure of the equilibria, in this paper we only consider the case of two players. The extension of the results to $N$ players seems feasible to us and is deferred to future research.}

Nonzero-sum stochastic differential games have been extensively studied, both in the regular control setting (see the classic references \cite{Basar1998,MR3752669}) and, more pertinently, in the singular control framework, see \cite{MR3832877, MR3770874, MR3914565, MR4332850, MR4395163}. The typical situation in these games is that, similarly to the control setting, the state-space is partitioned into a joint continuation (i.e. waiting) region, where both players are idle and let the system evolve uncontrolled, and two player-specific activation (i.e. installation) regions where only one agent exerts control with the other one remaining inactive. The novelty in our model is that we also allow both players to install simultaneously, hence we introduce a fourth region, of joint installation, where both players can install simultaneously.

The problem of renewable capacity expansion in the presence of competition has led to a growing body of literature in energy economics and mathematical finance. A variety of modelling frameworks have been employed to capture the inherent trade-offs. A significant branch of the literature addresses competitive renewable capacity expansion through Mean Field Game (MFG) frameworks, which analyse a continuum of small producers.
The paper \cite{MR4672784} studies the long-term development of renewable capacities, focusing on the cannibalization effect (price impact) of aggregate production and the role of subsidies, employing a deterministic price function.
The model in \cite{MR3715150} analyses the competition between technology subtypes (e.g., renewable vs. coal), and likewise utilises a deterministic clearing price which is a function of the aggregate capacity. More recent MFG models incorporate additional layers of uncertainty: \cite{escribe24} focuses on long-term investment under risk aversion and heterogeneous weather conditions, introducing exogenous common noise into the system, while \cite{hubert2025} models endogenous cannibalization effects, where the aggregate decisions of investors themselves influence equilibrium costs and prices. Although these models provide powerful tools for deriving macroscopic market properties and convergence results, they abstract from the direct strategic interaction and individual price impact that characterise a market with a small number of dominant players. 
From a different perspective, the paper \cite{agram2025} explores a computational methodology based on deep learning to address the optimal control of capacity expansion under jump uncertainty. 
Our setting is complementary to these works: it focuses on a small number of strategic agents whose irreversible installations directly affect a stochastic price process, and it relies on tools from singular control theory to derive a detailed equilibrium characterization.

As it turns out, without an a priori assumption on the structure of the continuation and installation regions of the two players, a verification theorem for the game we described above seems out of reach, starting from the very definition of solution to the HJB system, which seems sensitive to this structure. Therefore, in order to gain an intuition of 
the shapes of these regions, we first study a static, one-step, version of the game where the two players can only install at the beginning of the game, i.e. at $t=0$, when they can choose the amount of additional installed capacity, always subject to the constraint that the total capacity, i.e. $Y^1_0 + Y^2_0$, cannot exceed $\theta$. 
This characterizes a 3-dimensional state-space given by the Cartesian product between the whole real line (where the power price $x$ lies) and the 2-dimensional simplex of size $\theta$ (where installed capacities $y = (y_1,y_2)$ lie). Owing to the specific form of the objective functional, which is linear/quadratic with respect to variables $x/y$, respectively,
the resulting Nash equilibrium can be explicitly characterized. For a given initial power price $x$, the equilibrium strategy divides the 2-dimensional simplex into four regions: $\mathbb{W}^1\cap \mathbb{W}^2$, the waiting region common to both players; $\mathbb{I}^1\cap \mathbb{W}^2$ and $\mathbb{W}^1\cap \mathbb{I}^2$, the regions where only one of the players installs and the other one stays idle; and $\mathbb{I}^1\cap \mathbb{I}^2$, a fourth one that is the region of joint installation. Depending on the price level $x$, some of these regions may be absent from the corresponding section of the state-space, but in general one should take into account the presence of all four of them. While the players' strategy in the first three regions is analogous to other examples in the literature (i.e., each player either does not install, or installs enough power to reach the boundary of the common no-installation region), the novelty of this problem is the presence of the fourth region, where both players install. In particular, for fixed $x$, the section of $\mathbb{I}^1\cap \mathbb{I}^2$ is the intersection of a square {$( 0 , \frac{2}{3} A(x) )^2$} with the simplex. If this square is entirely contained in the simplex, then both players install enough capacity to steer their installation from their current state $y$ to the top-right corner of the square (see Figure \ref{fig:onestep_1} below). On the other end, if the top-right corner is outside of the simplex, then the equilibrium strategy is not unique, 
in particular any joint installation that saturate the system 
is a Nash equilibrium (see Figure \ref{fig:onestep_2} below). Another key feature of the solution is that the waiting region of Player $i$ is characterized as the region under the surface $x = F_i(y)$, with $F_i$ being a Lipschitz-continuous monotone function of the capacity, with the property of being constant with respect to the variable $y_j$ when {$y_i>y_j$} (see Figure \ref{fig:onestep_4}). As this particular structure will play a crucial role in the study of the dynamic game, it will be referred to as the \emph{$F$-structure} throughout the rest of the paper.  

With this in mind, we then tackled the dynamic, continuous-time, formulation of our game. Our main result, namely Theorem \ref{th_equilibrium}, can be roughly summarized as follows:
\begin{center}
\emph{``Any (\emph{regular}) solution to the associated variational system of PDEs coincides with the value functions of one, and only one, Nash equilibrium strategy."}
\end{center}
Critically, the correct definition of \emph{regular} solution (to the variational PDE) turns out to be the crux of the argument, and is tightly connected to the \emph{$F$-structure}, which is assumed to hold a priori. More precisely, the result succinctly stated above can be divided in two steps: 
\begin{itemize}
\item[(i)] the verification theorem, and 
\item[(ii)] the construction of the equilibrium strategies.
\end{itemize}
In Step (i), 
 we consider a system of free-boundary HJB PDEs, similar to the one obtained in \cite{MR4352705} with heuristic arguments. The first component of this system 
 can be formally written as
 \begin{equation}\label{HJB_bis_form}
\begin{cases}
\max\{(\cL-\rho)V_1(x,y)+xy_i,\, \p_{y_1}V_1(x,y)-c\}=0, & \text{on }\mathbb{W}^2 \\  
\p_{y_2}V_1(x,y) = 0, & \text{on }\mathbb{I}^2 =  (\mathbb{W}^2)^c 
\end{cases}, 
\end{equation}
where $V_1$ represents the value function of Player 1. On the domain $\mathbb{W}^2$, we have a variational inequality involving a transport operator in the variable $y_1$ and a second order operator in the $x$ variable, i.e. $\cL$, which is the generator of the OU-type process $X_t$. On $\mathbb{I}^2$, we have a transport operator in the variable $y_2$. The situation is totally symmetric for the second component of the system, namely the equation for the value function of Player 2. As previously mentioned, the chances of proving a verification theorem for the system \eqref{HJB_bis_form} are scarce unless one requires some additional structure for the solutions. In analogy with the static game, we then impose that the areas $\mathbb{W}^i$ (and thus $\mathbb{I}^i$), $i=1,2$, are given according to the \emph{$F$-structure} that we explicitely distinguished in the one-step game. Essentially, this allows us to see the variational inequality in \eqref{HJB_bis_form} as a free boundary problem, thus making the equations for $V_1$ and $V_2$ a free-boundary system. More precisely, the \emph{$F$-structure} is encoded in our definition of \emph{regular solution} to the variational PDE (cf. Definition \ref{def:regular_sol_VP}), and eventually in the definition of equilibrium solution (cf. Definition \ref{def:equilibrium_sol}) to the system. The \emph{$F$-structure} also allows us to determine the characterizing behavior of the optimal controls (cf. Definition \ref{def:equilibrium_sol}), and thus the behavior of equilibrium strategies. At this stage, we also had to account for the non-uniqueness of the equilibrium strategies in the region of joint installation, which arose in the one-step game. In particular, we selected a preferred strategy at saturation, by assuming that the players never install when their current level of installed power is greater than or equal than $\frac{\theta}{2}$. With this structure at hand, we were able to prove a verification theorem (cf. Theorem \ref{th:verific_th}), which roughly states that: for a given equilibrium solution to the HJB system, 
any strategy that yields optimal controls for both players, with respect to their respective free boundaries, is a Nash equilibrium. 

We stress that the verification theorem is proved under minimal regularity assumptions on the value function. In particular, our definition of regular solution to the variational PDE imposes (cf. Definition \ref{def:regular_sol_VP}-(i),(ii)), for the value function $V_1$: $C^2$ regularity in the $x$ variable only on $\mathbb{W}_2$, $C^1$ regularity in the $y_1$ variable only on $\mathbb{W}_2$, and $C^1$ regularity in the $y_1$ variable only on a lower neighborhood of $\mathbb{I}_2$. By symmetry, analogous regularity is required on $V_2$. This is the type of regularity that is expected in these types of singular problems (see \cite{cai2023change}), which typically translates into the technical hassle of applying the It\^o formula 
 to a function that is of class $C^{2,1}$ (namely $C^2$ in $x$ and $C^1$ in $y$) only on a closed domain with reflecting boundary. In our proof, we solve this problem via regularization, in both variables $x$ and $y$, a crucial step in passing to the limit being the result that optimal controls make the free boundary attainable with zero probability at each time (cf. Lemma \ref{lem:boundary}).

Concerning Step (ii), given an equilibrium solution to the variational system, we construct the only strategy that yields optimal controls for both players, relative to the free boundaries $F_i$, which is then a Nash equilibrium (cf. Theorem \ref{th_equilibrium}) in light of the verification theorem. Once more, we heavily rely on the \emph{$F$-structure}, which translates the problem of constructing optimal controls into a Skorokhod-type problem. More precisely, we need to build reflected diffusions with respect to the free boundaries, which are the $2$-dimensional surfaces $x = F_i(y)$, $i = 1,2$. Recall that, in our problem, the state-space is three-dimensional, with the diffusion only acting in the $x$ variable. The resulting Nash equilibrium strategies are analogous to the optimal strategies obtained in \cite{MR4305783} in the one-player case. In particular, they can possibly jump only at the initial time, when the initial state is out of the joint continuation region, with a lump installation that steers $(X_0, Y_0)$ to the boundary of $\mathbb{W}^1 \cap \mathbb{W}^2$. For any positive time, instead, the equilibrium strategy is continuous. 

Assuming, for instance, that $y_1<y_2$, namely the initial capacity of Player 1 is lower than the one of Player 2, and that the initial electricity price $x$ is low enough so that $(x,y_1,y_2)\in \mathbb{W}^1 \cap \mathbb{W}^1$, which is equivalent to having $x < F_1(y)$ in light of the \emph{$F$-structure}, the equilibrium strategy works as follows:
\begin{itemize}
\item[-] Initially, both players are idle. Their installed capacities $Y_t =( Y^1_t , Y^2_t)$ remain constant and the electricity price $X_t$ follows an uncontrolled OU process. 
\item[-] When $X_t$ is sufficiently high so that the triple $(X_t, Y^1_t , Y^2_t)$ hits the free boundary of Player 1, namely $X_t = F_1 (Y_t)$, Player 1 installs, continuously, the minimum amount to keep $X_t \leq F_1 (Y_t)$. When $X_t$ is again smaller than $F_1(Y_t)$, Player 1 returns idle, until the boundary of Player 1 is hit again. Note that Player 2 has been idle all along.
\item[-] If the initial capacity of Player 2 were higher than the half-total capacity, i.e. $y_2 \geq \theta/2$, then the two previous regimes keep on repeating until full saturation is reached, i.e. $Y^1_t = \theta - y_2$. On the other hand, if $y_2 < \theta/2$, they repeat so long as $Y^1_t < Y^2_t$.
\item[-] After $Y^1_t = Y^2_t$ for the first time, the two capacities will remain equal forever after. In other words, the two players will follow the same installation strategy and process $Y_t$ will never leave the angle bisector of the first quadrant. In particular, both players will be idle whenever $X_t < F(Y_t) := F_1(Y_t) = F_2(Y_t)$, while they will be active whenever $X_t  = F(Y_t)$ by installing, continuously, the minimum amount to keep $X_t \leq F(Y_t)$.
\item[-] Eventually, the capacity will reach saturation with $Y^1_t = Y^2_t = \theta/2$ and the players will remain idle forever after.
\end{itemize}


The next, natural, step in the analysis would be the study of the solutions to the variational system, namely the existence and uniqueness of its solutions, possibly via (semi-)explicit construction of the value function and the free-boundary, as it was done in \cite{MR4305783} for the single-player case. While this task is the subject of ongoing research, here we draw some preliminary conclusions on the structure of the value function, as a consequence of the regularity requirements discussed above. In particular, we derive a system of three \emph{smooth-fit} conditions (cf. \eqref{eq:system}) that ought to be satisfied on the free-boundary. 

\medskip

The rest of the paper is organized as follows. In Section 2, we introduce the general problem of finding Nash equilibria between the two players, with all the relevant definitions. In Section 3, we study the one-step game obtained by restricting the set of admissible strategies to those that only allow for one initial installation, and derive the corresponding Nash equilibria. 
In Section 4, we go back to the general continuous-time game: in Section \ref{sec:ver}, we introduce the notion of equilibrium solution to the variational system and we prove the verification theorem (Theorem \ref{th:verific_th}), together with the required preliminary results, such as {\color{red} a} weak Ito formula (Lemma \ref{lem:ito}); in Section \ref{sec:construc}, we construct the equilibrium strategy associated to a given solution to the variational system and prove the main result (Theorem \ref{th_equilibrium}); in Section \ref{sec:regularity_V_discus}, we discuss the implications that the regularity assumptions made on the solution of the variational PDE have on the structure of the value function.

\section{A market with two players}\label{sec:game}

Let {$W=(W_t)_{t\geq 0}$} be a standard one-dimensional Brownian motion on  a complete filtered probability space 
$(\Omega,\mathcal{F},P, \{\mathcal{F}_t\}_{t\ge 0})$. Also denote with $x$ a point in $\R$ and with $y=(y_1,y_2)$ a point in $\R^2$. We assume that, in absence of any companies' economic activities, the electricity price $(X^x_t)_{t\ge 0}$ evolves according to a mean-reverting Orstein-Uhlenbeck dynamics
\begin{equation}\label{eq:OU_SDE}
dX^{x}_t=k\left(\mu-X^{x}_t\right)dt+\sigma dW_t, \qquad X^x_0=x,
\end{equation}
for some constants {$k, \sigma >0$} and $\mu\in \R$. 

We consider a market where two producers, indexed by $i=1,2$, operate. {Hereafter, we denote by $0^{-}$ an infinitesimal negative time, precisely we define the ordered set $[0^{-},+\infty):=\{0^-\} \cup [0,+\infty)$ with $0^{-}<t $ for any $t\geq 0$.}
The installed power level of producer $i$ is described by the process 
\begin{equation}\label{installed_power}
Y^{x,I,i}_t=y_i+I^i_t, \qquad t\ge 0^-,
\end{equation}
where $y_i$ is the initial level of installed power, $I_t=(I_t^1,I_t^2)$ is the total power installed by the two producers {at time $t$}, 
and each component $I^i_t$ identifies the $i$-th company's control variable, with $I^i_{0-} = 0$, $i = 1,2$. {We consider the case of irreversible installation, and thus the process $I^i$ (and $Y^i$) are nondecreasing.} We also assume that the total installed power level cannot exceed a certain threshold $\theta$. {This constraint might be motivated by different reasons, e.g. only a finite real estate is available for the installation of solar panels or wind power plants. Mathematically speaking, this gives rise to a so called \emph{finite-fuel} singular stochastic control framework: this is the main motivation for introducing a new state variable $(Y_t)_t$ that coincides, up to fixing the initial value, with the control $(I_t)_t$.}  Accordingly, {for a given $y\in D$, with  
\begin{equation}
D:=\{y=(y_1,y_2)\in\R^2, \; y_i\ge 0, \; y_1+y_2<\theta\},
\end{equation}}
the set of \textit{admissible controls} is given by 
\begin{equation}\label{eq:def_strat_ammiss}
\cI (y) := \left\{ I:{[0^-,\infty)}\times \Omega \mapsto [0,\infty)^2, \text{non-decreasing, cadlag}, 
\; {I_{0^-} = (0,0)\text{ and }}\sum_{i=1}^2(y_i+I^i_t)\le \theta \right\}.
\end{equation}
Notice that each player is constrained, in its strategy, by the installation strategy of the other.
Following  \cite{MR4305783}, we assume that the current total level of electricity production, which is proportional to $\sum_{i=1}^2Y^i_t$, has a negative effect on the electricity price: {its long-term mean parameter is reduced, at time $t$, by $\beta \sum_{i=1}^2Y^i_t$}, for some $\beta>0$.
Therefore, the spot price $(X^{x,y,I}_t)_{t\ge 0}$ evolves according to 
\begin{align}\label{model}
\begin{cases}
dX^{x,y,I}_t=k\left(\mu-\b \sum_{i=1,2}Y_{t}^{y,I,i}-X^{x,y,I}_t\right)dt+\sigma dW_t, \qquad t>0,\\
X^{x,y,I}_0=x.
\end{cases}
\end{align}
{
\begin{remark}\label{eq:sde_well_posed}
Given $(x,y)\in \R \times D$ and $I\in \cI (y)$, the stochastic system \eqref{installed_power}-\eqref{model} is strongly well-posed, meaning that there exists a pathwise unique strong solution to \eqref{installed_power}-\eqref{model}.
\end{remark}
}
In this setting, each company aims at maximizing the profit, derived from selling electricity in the market, which is described, for any admissible control $I$, by the 
utility functional 
\begin{equation}\label{cost}
\cS_i(x,y,I)=\mathbb{E}\left[\int_{0}^{\infty}e^{-\rho\t}\,X^{x,y,I}_{\tau}\, Y^{y,I,i}_{\tau}\, d\t-c\int_{{0^-}}^{\infty}e^{-\rho\t}dI^i_{\tau}\right],
\end{equation}
where {$\rho>0$} is a discount factor and $c\ge 0$ is a constant representing the cost of installing one unit of power. 

\subsection{Problem formulation}

In \cite{MR4352705}, the authors consider {a cooperative situation motivated by the presence of a social planner}, where the problem consists of finding an efficient installation control ${I\in \cI (y)}$ which maximizes the 
aggregate expected profit, net of investment cost. This is known as a Pareto optimum and is expressed as
$$\sup_{{I\in  \cI (y)}}\cS_{SP}(x,y,I),\qquad  \cS_{SP}(x,y,I):=\cS_1(x,y,I)+\cS_2(x,y,I).$$
 Setting $\bar{y}=y_1+y_2$, $\bar{I}_{\tau}=I^1_{\tau}+I^2_{\tau}$ {and $\bar{Y}_t=Y^1_t+Y^2_t$}, it is straightforward to check that 
\begin{equation}
\cS_{SP}(x,y,I)=\mathbb{E}\left[\int_{0}^{\infty}e^{-\rho\t}\,X^{x,y,I}_{\tau}\, {\bar{Y}_{\tau}}\, d\t-c\int_{0}^{\infty}e^{-\rho\t}d\bar{I}_{\tau}\right].
\end{equation}
Then the aggregate optimal strategy for the social planner is equivalent to the optimal control for a single player, as investigated in \cite{MR4305783}. Such strategy does not characterize the single installations $I^i$, and it is indeed not unique. Notice as well that the problem can be formulated for $N$ players without any  modification. 

In this paper we look for an equilibrium solution to the competitive game between the two companies, also known as Nash equilibrium. Generally speaking, a Nash equilibrium is {a locally} optimal control for both players, in the sense that neither one can improve their profit by changing their control while the other keeps theirs unchanged. {We introduce a \emph{closed loop}-type notion of Nash equilibrium that is suitable for singular stochastic control problems.
We first need to introduce some additional notations as well as the notion of \emph{admissible strategy}.}

\begin{notation}
For a given $t>0$, we denote by $C([0,t])$ and $\NCL{t}$ the sets of real-valued continuous functions and non-decreasing {$D$-valued} c\`adl\`ag functions, respectively, defined on $[0,t]$.
\end{notation}

\begin{definition}[Admissible strategy]\label{def:markov_strategy}
For a given $(x,y)\in \R\times D$, we call \emph{admissible strategy} a family $\contr = (\contr_t)_{t\geq 0}$ of Borel-measurable functions $\contr_t:  C([0,t])\times \NCL{t} \to \R^2$ such that the control $I = (I_t)_t$, defined by
\begin{equation}\label{eq:contr_markov}
I_t = \contr_t\big((X^{x,y,I}_s)_{s\leq t}, (Y^{y,I}_{s})_{s\leq t}\big), \qquad t\geq 0.
\end{equation}
belongs to $\mathcal{I}(y)$. The set of all admissible strategies for the state $(x,y)\in \R\times D$ is denoted by $\mathcal{A}^{(x,y)}$. 
\end{definition}
{Note that the set of admissible strategies $\mathcal{A}^{(x,y)}$ is well-defined in light of Remark \ref{eq:sde_well_posed}. Also, in the cases when there is no risk of ambiguity, we will omit the superscript on $\mathcal{A}^{x,y}$  and simply write $\mathcal{A}$. 
\begin{definition}[Nash equilibrium]\label{def:nash_equilibrium} {For a given $(x,y)\in \R\times D$, a}  function $\contr \in \mathcal{A}$ is called a \emph{Nash equilibrium} if, for any $\{i,j\}$ permutation of $\{1,2 \}$ and $(x,y)\in\R\times D$ we have:
\begin{equation}\label{eq:nash_equilibrium}
\cS_i(x,y,I)\ge \cS_i(x,y,I'), 
\end{equation}
with {$I$ being the only control in $\mathcal{I}(y)$ satisfying in \eqref{eq:contr_markov}}, for any $I'\in\mathcal{I}(y)$ such that 
\begin{equation}
I'^j_t = \contr^{j}_t\big((X^{x,y,I'}_s)_{s\leq t}, (Y^{y,I'}_{s-})_{s\leq t}\big), \qquad t>0.
\end{equation} 
\end{definition}
{\begin{remark}
As our approach to the problem is based on the study of the associated HJB system and its related verification theorem, it is crucial that the notion of Nash equilibrium is in feedback form (i.e. closed-loop). However, we stress that the suboptimal strategies $I'$ in \eqref{eq:nash_equilibrium} are of mixed type, meaning that the $j$-th player's control $I'^j$ is in feedback-form, but the $i$-th player can employ any control $I'^i$ as long as the pair $I'=(I'^i,I'^j)$ is admissible. This is stronger than requesting \eqref{eq:nash_equilibrium} to hold for any $I' = \contr'_{\cdot}\big((X^{x,y,I}_s)_{s\leq \cdot}, (Y^{y,I}_{s})_{s\leq \cdot}\big)$ with $\contr' \in \mathcal{A}$.
\end{remark}
\begin{remark}
An open-loop type notion of Nash equilibrium can be defined as a control $I\in \mathcal{I}(y)$ such that, for any $\{i,j\}$ permutation of $\{1,2 \}$ and $(x,y)\in\R\times D$ we have:
\begin{equation}\label{eq:nash_equilibrium_open}
\cS_i(x,y,I)\ge \cS_i(x,y,I'), 
\end{equation}
for any $I'\in \mathcal{I}(y)$ such that $I'^j = I^j$. Note that there are at least two (non-symmetric) trivial Nash equilibria of this type, given by 
\begin{equation}
I^1_t = \bar I_t \text{ (or $I^1_t = 0$)}, \qquad I^2_t = 0\text{ (or $I^2_t = \bar I_t$)}, \qquad t\geq 0,
\end{equation}
with $\bar I$ being the optimal strategy in \cite{MR4305783} for the one-dimensional problem starting from the initial installed power $\bar y = y_1 + y_2$. Indeed, if $I^2 \equiv 0$, then it is obviously optimal for Player 1 to install according to $\bar I$. On the other hand, in order for Player 1 to be able to implement $\bar I$, Player 2 should not install at any time
, and thus $I^2 \equiv 0$ is the only possibile control such that $(\bar I, I^2)\in \mathcal{I}(y)$. This stems from the finite-fuel nature of the control problem, in particular the fact that the total installation is bounded by $\theta$. Indeed, if Player 1 installs according to $\bar I$, then $Y^1_t = \theta-y_1-y_2$ almost surely for $t$ sufficiently large: this implies that $I^2\equiv 0$.
\end{remark}
}

\section{One-step game}\label{sec:one_step}

To develop an intuition on how a Nash equilibrium strategy should work and to introduce some key concepts for the subsequent analysis, we first analyze {a static game} where both players can install only at time $t=0$. In this situation, an admissible control {$I_t \equiv I=(I^1,I^2)$} (equivalently a strategy) 
is completely determined by the deterministic values 
\begin{equation}\label{eq:static_admissible}
I^i={\bf I}^i(x,y)\ge 0, \quad {i=1,2,}\qquad \text{s.t.}\quad \sum_{i=1}^2 (y_i+I^i)\leq \theta.
\end{equation} 
with ${\bf I}^i : \mathbb{R} \times D \to D$. 
{As the solution to \eqref{model}, in this case, is simply an Orstein-Uhlenbeck process, the expected value in \eqref{cost} can be trivially computed. In particular, the expected profit reads as} 
\begin{equation}\label{eq:profit_onestep}
\mathcal{S}_i(x,y,I)=(y_i+I^i)\frac{x\rho+\mu k-\beta k \langle {\bf 1}, y+I \rangle}{\rho (\rho+k)}-c\, I^i,
\end{equation}
with ${\bf 1} := (1,1)$, 
which is a quadratic function of $I^i$. Then the equilibrium condition \eqref{eq:nash_equilibrium} of Definition \ref{def:nash_equilibrium} can be  expressed as follows. 
{\begin{definition}[Nash equilibrium for the one-step game]
A function ${\bf I}:\R\times D \to \R^2$ is a Nash equilibrium if 
the constraints \eqref{eq:static_admissible} are satisfied and, for any $\{i,j\}$ permutation of $\{1,2 \}$, we have
\begin{equation}
\mathcal{S}_i(x,y,I) \geq \mathcal{S}_i(x,y,I'), \qquad (x,y)\in\R\times D,
\end{equation}
with $I' = {\bf I}'(x,y)$, for any ${\bf I}':\R\times D \to \R^2$ satisfying the constraints \eqref{eq:static_admissible} and such that ${\bf I}'^j={\bf I}^j$.
\end{definition}
More explicitly, a function ${\bf I}:\R\times D \to \R^2$ satisfying the constraints \eqref{eq:static_admissible} is a Nash equilibrium if
\begin{equation}\label{Nash_onestep}
(I'^i-I^i)(2A(x)-y_j - I^j-I'^i-I^i)\le 0,  \qquad  I'^i\in [0,\theta-y_i-y_j-I^j],
\end{equation}
for any $(x,y)\in\R\times D$, where $A(x)$ is the linear function of the initial spot price given by
\begin{equation}\label{A(x)}
A(x):=\frac{x\rho+\mu k-c\rho (\rho+k)}{2\beta k}.
\end{equation}}

\subsection{Equilibrium strategies}
{Fix hereafter $(x,y)\in\R\times D$. For any $\{i,j \}$ permutation of $\{1,2\}$, a}ssume that {Player $j$} chooses the 
strategy {$I^j = {\bf I}^j (x,y)$}. Then we look for a strategy $I^i$ given by
\begin{equation} \label{argmax}
I^i = \argmax_{I'^i\in [0,{\theta-y_i-y_j-I^j}]} \mathcal{S}_i(x,y,I'^i,I^j).
\end{equation}
{Therefore, as $\mathcal{S}_i(x,y,\cdot, I^j)$ is a convex quadratic function, by simply differentiating \eqref{eq:profit_onestep}
 it is easy to check that the set of the Nash equilibria is given by the pairs $I=(I^1,I^2)\in\R^2$ such that
\begin{equation}\label{eq:onestep_1}
\begin{cases}
I^1=\min\{\theta{-y_2-I^2}-y_1, A(x)-\frac{1}{2}{(y_2+I^2)}-y_1\}\vee 0\\
I^2=\min\{\theta{-y_1-I^1}-y_2, A(x)-\frac{1}{2}{(y_1+I^1)}-y_2\}\vee 0
\end{cases}.
\end{equation}
}
This means that we can characterize a {Nash equilibrium $I=({I^1},I^2)$} by the geometric properties of the locus of points {of the expressions in} \eqref{eq:onestep_1}, depending on $A(x)$, which is a linear function of the price. Let us first consider the intersection of the lines 
\begin{align}
P(y)&={\left\{(z_1,z_2)\in \R^2, \; z_1=A(x)-\frac{1}{2}(y_2+z_2)-y_1, z_2=A(x)-\frac{1}{2}(y_1+z_1)-y_2 \right\} }\\
&=\left(\frac{2}{3}A(x)-y_1,\frac{2}{3}A(x)-y_2 \right), \label{eq:onestep_2}
\end{align}
which corresponds to the solution to the optimization problem in the absence of constrains. 
By \eqref{eq:onestep_1} we can make the following preliminary observations: 
\begin{itemize}
\item[-] If $\frac{2}{3}A(x)-y_i\le 0$ {for $i=1,2$, then $I^1=I^2=0$, meaning neither player does install}.
\item[-] If $\frac{2}{3}A(x)-y_2\le 0$ and $\frac{2}{3}A(x)-y_1> 0$, then {$I^2=0$ and} by \eqref{eq:onestep_1}, $I^1=\min \{\theta-y_1-y_2,A(x)-y_1-\frac{y_2}{2}\}\vee 0$. 
{As $y\in D$,} this means that {Player $1$} installs {only} when $A(x)-y_1-\frac{y_2}{2}> 0$, and saturates the capacity of the system if 
$${\theta-y_1-y_2\le A(x)-y_1-\frac{y_2}{2}\quad \Longleftrightarrow \quad A(x)+\frac{y_2}{2}\ge \theta}.$$
{Viceversa if $\frac{2}{3}A(x)-y_1\le 0$ and $\frac{2}{3}A(x)-y_2> 0$.} 
\item[-] If $\frac{2}{3}A(x)-y_i >  0$ {for $i=1,2$, then both players do install (as $y\in D$)}. 
\end{itemize}
%

Let us now visualize both players' strategies, depending on the {initial electricity price $x$.} 
We may distinguish four cases: 
\begin{itemize}
\item[i)]$A(x)\le 0$: in this situation $\frac{2}{3}A(x)-y_i {\leq} 0$ for both $i=1,2$ {(as $y\in {D}$)}, therefore $I^1=I^2=0$. Roughly speaking, the electricity price is too low to justify any installation. 
\item[ii)]$0< A(x) \leq \frac{\theta}{2}$: in this price interval one or both players may install, depending on the initial condition $y\in {D}$. However the system will never reach saturation after any installation: indeed, let us assume for instance that $I^1>0$, and suppose by contradiction that the system reaches the maximum capacity. Then clearly
\begin{align*}
I^1&=\theta-y_2-I^2-y_1
\intertext{(by the first equation in \eqref{eq:onestep_1})}
&\le A(x)-\frac{1}{2}(y_2+I^2)-y_1.
\end{align*}
Adding $y_1$ on both sides of the inequality and exploiting the price regime $A(x)\leq \frac{\theta}{2}$, we get 
\begin{align*}
\theta-y_2-I^2 &\le A(x)-\frac{1}{2}(y_2+I^2)\le  \frac{1}{2}(\theta-y_2-I^2),  
\end{align*}
which yields $\theta-y_2-I^2=0$. Therefore $y_1=I^1=0$, in contradiction with the assumption $I^1>0$. An analogous situation verifies if $I^2>0$.

{In particular:}
\begin{itemize}
\item[ii-a)] If $\frac{2}{3}A(x)-y_i> 0$ for both $i=1,2$, i.e. the intersection point \eqref{eq:onestep_2} corresponds to an admissible strategy, then
$$I^i=\frac{2}{3}A(x)-y_i,  \qquad i=1,2.$$ 
Notice that, regardless on the initial level of installed power, both players reach the same level {(exactly $\frac{2}{3}A(x)$)} after installation. {We will comment again on this {below, in Section \ref{sec:free_bound_static}}.}
\item[ii-b)] If $\frac{2}{3}A(x)-y_2\le 0$ and $\frac{2}{3}A(x)-y_1> 0$, by \eqref{eq:onestep_1} we have $I^2=0$, and 
\begin{equation}\label{eq:onestep_4}
I^1=\big( A(x)-\frac{1}{2}y_2-y_1 \big)\vee 0,
\end{equation} 
{with Player $1$ reaching level $(A(x)-\frac{1}{2}y_2) \vee y_1$ after installation. Viceversa for $\frac{2}{3}A(x)-y_1\le 0$ and $\frac{2}{3}A(x)-y_2> 0$}
\item[ii-c)] Lastly, if $\frac{2}{3}A(x)-y_i\le 0$ for both $i=1,2$, then $I^1=I^2=0$, {meaning that neither player installs.}
\end{itemize}
\item[iii)] $\frac{\theta}{2}<A(x)\le \frac{3}{4}\theta$: this case is {the same as ii) if $y_i \leq 2(\theta-A(x))$ for both $i=1,2$. 
If $y_2 > 2(\theta-A(x))\ge \frac{2}{3}A(x)$, we have instead}
$$A(x)-\frac{1}{2}y_2>\theta-y_2, $$
and therefore {Player $1$} must install up to saturating the total capacity $\theta$. Precisely we have 
\begin{equation}\label{eq:I_case_iii}
I^1=\theta-y_1-y_2, \qquad I^2=0,
\end{equation}
{
with Player 1 reaching the level $\theta - y_2$. Viceversa, for $y_1 > 2(\theta-A(x))$, the conclusions above hold true by swapping the subscripts $1$ and $2$.}
See Figure \ref{fig:onestep_1} for a visual representation of cases ii) and iii).
\begin{figure}[!h]
\centering
\includegraphics[scale=0.45]{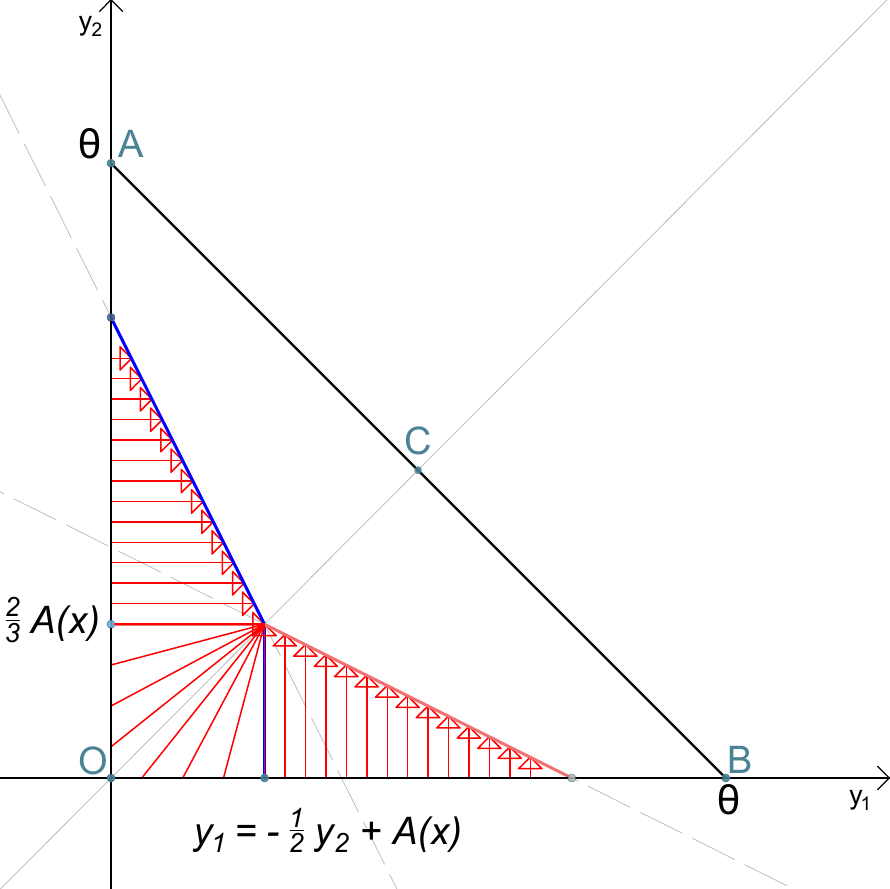} \quad
\includegraphics[scale=0.43]{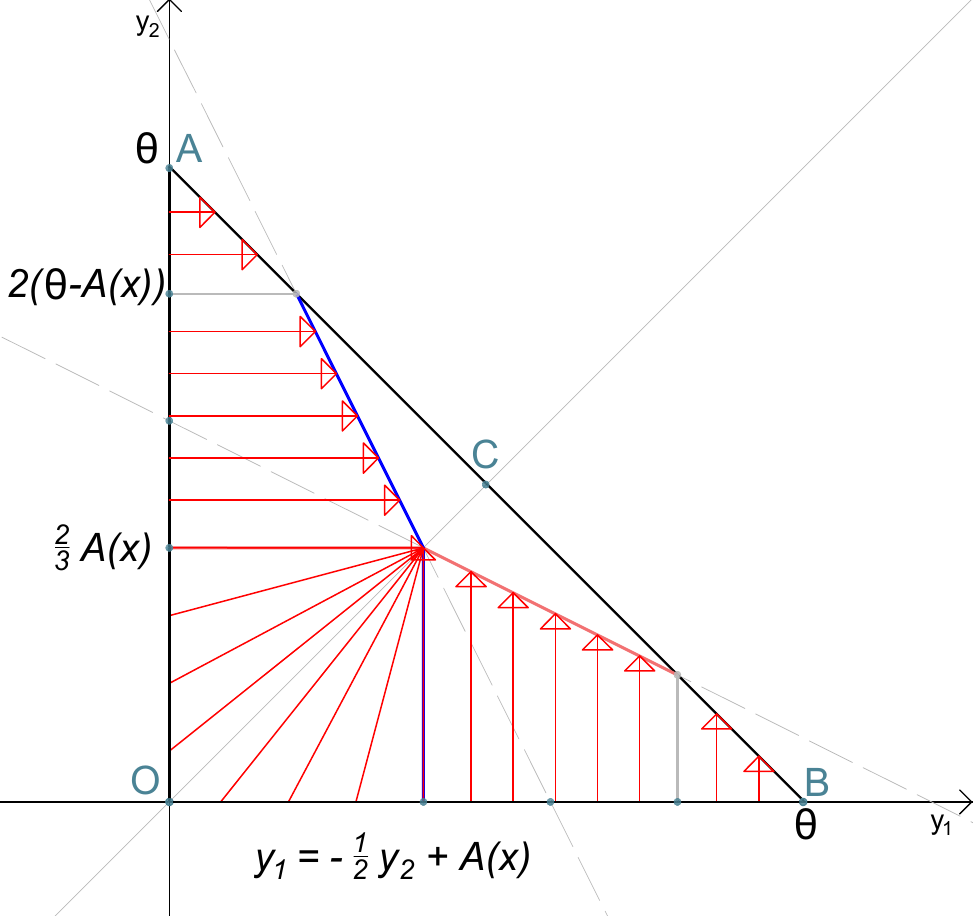}
\caption{{Equilibrium strategies in cases  ii) and iii): only Player 1 installs (${\red \to}$); only Player 2 installs (${\red \uparrow}$); both players install (${\red \nearrow}$); neither player installs (blank).}
}\label{fig:onestep_1}
\end{figure}
\item[iv)]
$A(x)>\frac{3}{4}\theta$: for this price level, the system always reaches a saturation state. In particular: 
\begin{itemize}
\item[iv-a)]
If $y_2>\frac{2}{3}A(x)$, then {$I$ is like in \eqref{eq:I_case_iii}, in particular Player 1 saturates the total capacity $\theta$. Viceversa if $y_2>\frac{2}{3}A(x)$. 
\item[iv-b)] If $y_i\leq\frac{2}{3}A(x)$ for both $i=1,2$, then} it is straightforward to check that {the set of the pairs $I=(I^1,I^2)\in\R^2$ such that
\begin{equation}\label{eq:onestep_1_bis}
\begin{cases}
I^1=\min\{\theta{-y_2-I^2}-y_1, A(x)-\frac{1}{2}{(y_2+I^2)}-y_1\} \\
I^2=\min\{\theta{-y_1-I^1}-y_2, A(x)-\frac{1}{2}{(y_1+I^1)}-y_2\}
\end{cases}
\end{equation}
}
corresponds to{
\begin{equation}
\left\{I=(I^1,I^2)\in\R^2 ,\, y_1+I^1+ y_2 + I^2=\theta, \, I^1\ge 2 (\theta-A(x))-y_1, \, I^2\ge 2 (\theta-A(x))-y_2 \right\}.
\end{equation}}
{A visual representation is given in} Figure \ref{fig:onestep_2bis}.
\begin{figure}[!h]
\centering
\includegraphics[scale=0.45]{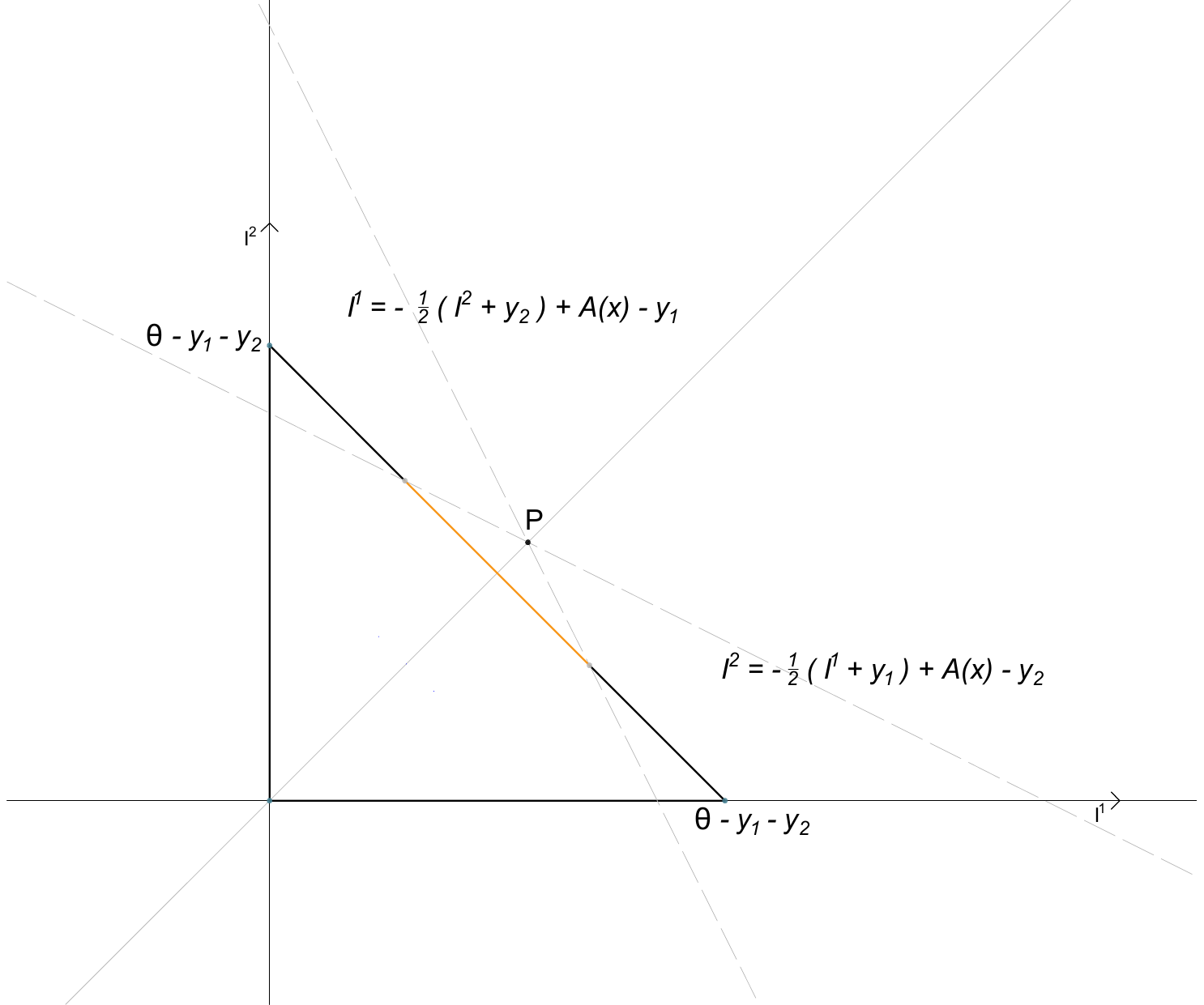}
\caption{{Case iv-b): all the Nash equilibria lie on the yellow segment.}
} \label{fig:onestep_2bis}
\end{figure}
This means that \textit{any} admissible strategy which saturates the capacity of the system is an equilibrium strategy. {In particular, these are infinitely many.}
{In Figure \ref{fig:onestep_2} we show two possible strategies at saturation. The strategy \textbf{a} appears to be more regular with respect to $x$, while on the other hand, strategy \textbf{b} is easier to implement. Hereafter will set our analysis to match the latter idea, {which is easier to generalize to the dynamic game}. }
\begin{figure}[!h]
\centering
\includegraphics[scale=0.45]{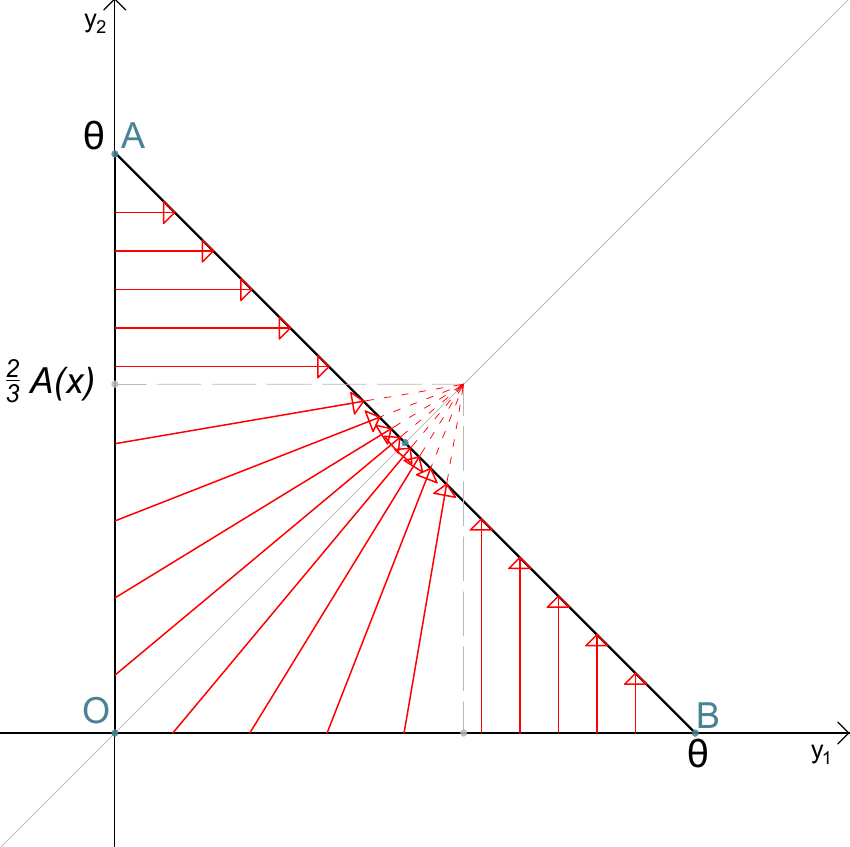}\quad
\includegraphics[scale=0.4]{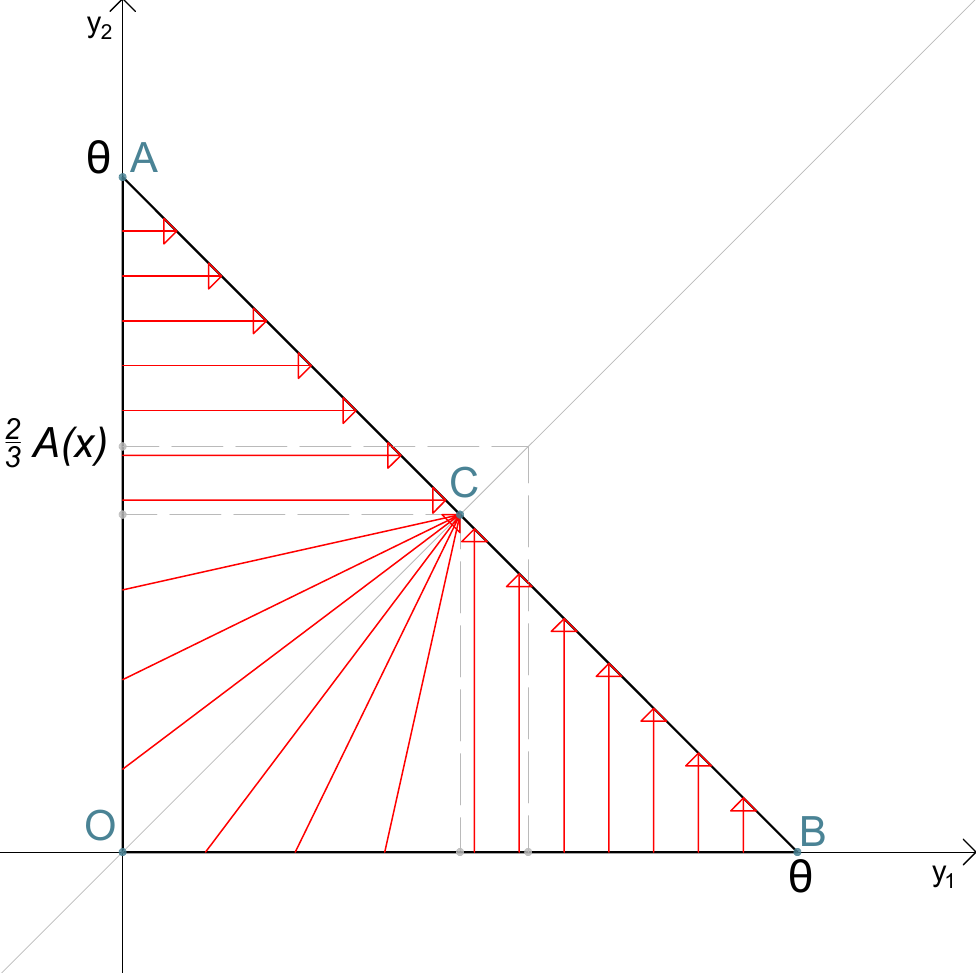}
\caption{{Equilibrium strategies in case iv-b), strategies  \textbf{a} (left) and \textbf{b} (right): only Player 1 installs (${\red \to}$); only Player 2 installs (${\red \uparrow}$); both players install (${\red \nearrow}$). 
}
} \label{fig:onestep_2}
\end{figure}
\end{itemize}
\end{itemize}
{The reader can easily check \eqref{eq:onestep_1} through cases i) to iv)}. 

\subsection{Characterization of the free boundaries: the \emph{$F$-structure}}\label{sec:free_bound_static}

Let us examine Figure \ref{fig:onestep_1}{, corresponding to cases (ii) and (iii) listed above, namely $0< A(x) \leq \frac{\theta}{2}$ and $\frac{\theta}{2}<A(x)\le \frac{3}{4}\theta$}, respectively. The {blue segments} 
can be interpreted as the level set of a boundary function $F_1:{D}\mapsto \R$, which separates two regions relative to Player 1, that we might call:
\begin{itemize}
\item the \textit{installation region} $\mathbb{I}^1$, composed by all the points $(x,y)$ for which, at Nash equilibrium, the optimal installation $I^1$ defined in Equation \eqref{argmax} is strictly positive, and 
\item the \textit{waiting region} $\mathbb{W}^1 := (\mathbb{R} \times D) \setminus \mathbb{I}^1$. 
\end{itemize}
Similarly, the red segments constitutes the level set of a boundary function $F_2$, separating Player 2's installation region $\mathbb{I}^2$ and waiting region $\mathbb{W}^2$. For now, these regions are intuitively described as above, while the formal definitions (where for example we will require $\mathbb{I}^i$, $i = 1,2$, to be closed sets for technical reasons) will be given later. Clearly $F_1(y_1,y_2)=F_2(y_2,y_1)$ by the symmetry of the game.
%

In order to extend the computation above to the dynamic case, it is crucial to understand the structure of these regions and {their} boundaries. 
For any {$y\in \overline{D}\cap \{y_1\le y_2\}$} we set {
\begin{equation}\label{eq:onestep_freebound}
{F}(y):=A^{-1}\left(\frac{y_1}{2}+y_2\right),
\end{equation}
and 
\begin{equation}
F_2(y_1,y_2) := F(y_1,y_2), \qquad F_1(y_1,y_2) := F_2(y_2,y_1).
\end{equation}}
\begin{remark}\label{rem:increasing_F}
As $A$ is an increasing function, the boundary function $F_1 = F_1(y_1,y_2)$ defined above is increasing in both variables. The same holds true for $F_2 = F_2(y_1,y_2) = F_1(y_2,y_1)$.
\end{remark}
Then we can characterize the waiting and installation regions for {Player $2$} as follows:
{\begin{equation}\label{eq:onestep_wait}
\mathbb{W}^2:=W^{\text{free}}_{{2}}\cup W^{\text{prol}}_{{2}}\cup W^{\text{sat}}_{2}, \qquad \mathbb{I}^2:=(\R\times  {D} )\setminus \mathbb{W}^2.
\end{equation}
with
\begin{align*}
W^{\text{free}}_{{2}}&:=\left\{(x,y)\in \R\times {D}, \, y_1\geq y_2, \, x < F_2(y)\right\},\\
W^{\text{prol}}_{{2}}&:=\left\{(x,y)\in \R\times {D}, \, {y_1 \leq y_2 <  \frac{\theta}{2}}, \, x<F_1(y_1,y_1)\right\},\\
W^{\text{sat}}_{2}&:=\left\{(x,y)\in \R\times {D}, \, y_2\geq \frac{\theta}{2}\right\}.
\end{align*}}
Analogous, symmetric, definitions can be given for the areas $\mathbb{W}^1$ and $\mathbb{I}^1$. As this way of defining the regions $\mathbb{W}^i$ and $\mathbb{I}^i$ will play a crucial role in the study of the dynamic game, we will refer to it hereafter as the \emph{$F$-structure}. 
 The three regions that compose $\mathbb{W}^2$ can be interpreted as follows
\begin{remark}\label{rem:interpretation_W}
The choice of the suffixes ``free", ``prol" and ``sat" is motivated as follows. {$W^{\text{free}}_2$ is defined as the sub-graphic of the boundary function $F_2$, which, in the continuous-time dynamic game, will be understood as a free boundary of the HJB system. $W^{\text{prol}}_{2}$ is defined as the sub-graphic of the function given by the values $F_2(y_2,y_2)$ prolonged onto the points $(y_1,y_2)$. The component $W^{\text{sat}}_{2}$ matches the choice of strategy \textbf{b} in case iv): indeed player $2$ does not install if his starting level of installed power is higher than $\frac{\theta}{2}$, which can be understood as a saturation point for the single player.}
\end{remark} 
{In Figure \ref{fig:onestep_4} we may observe the typical boundary structure in the three dimensional space. }
\begin{figure}[!h]
\centering
\includegraphics[scale=0.38]{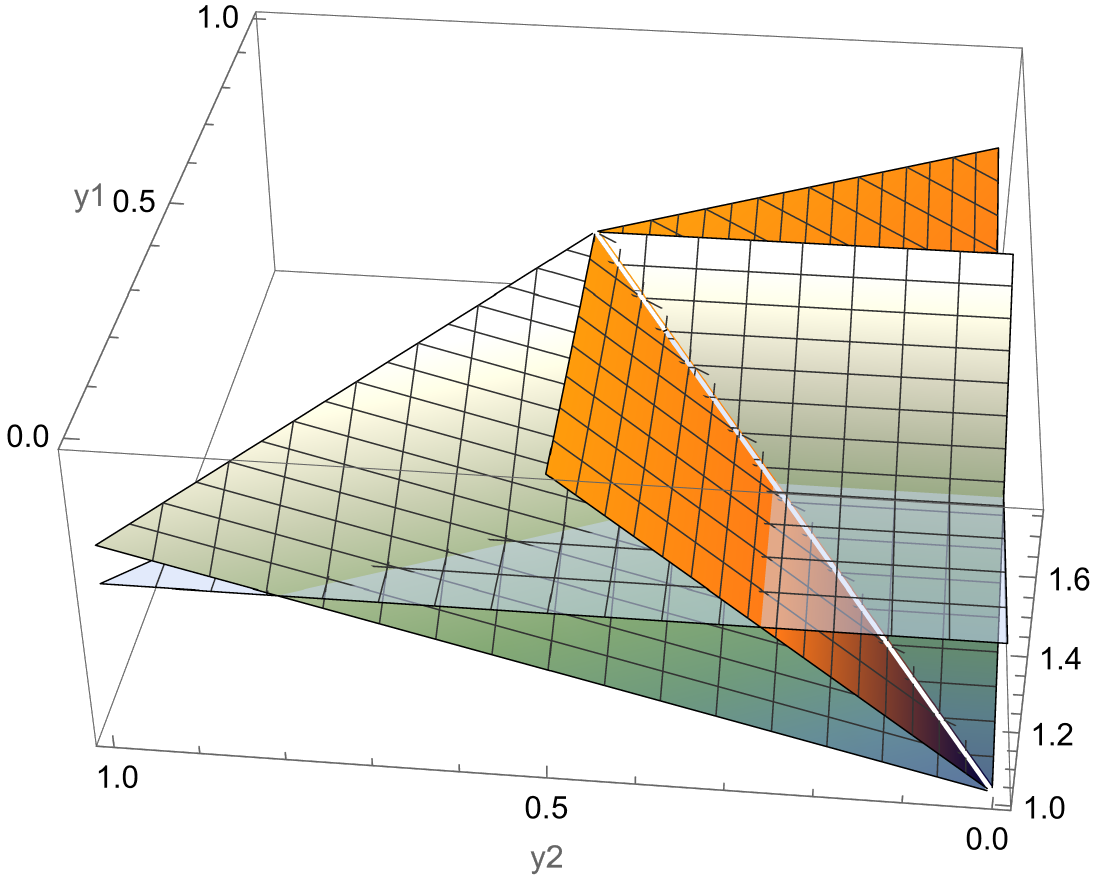}
\includegraphics[scale=0.28]{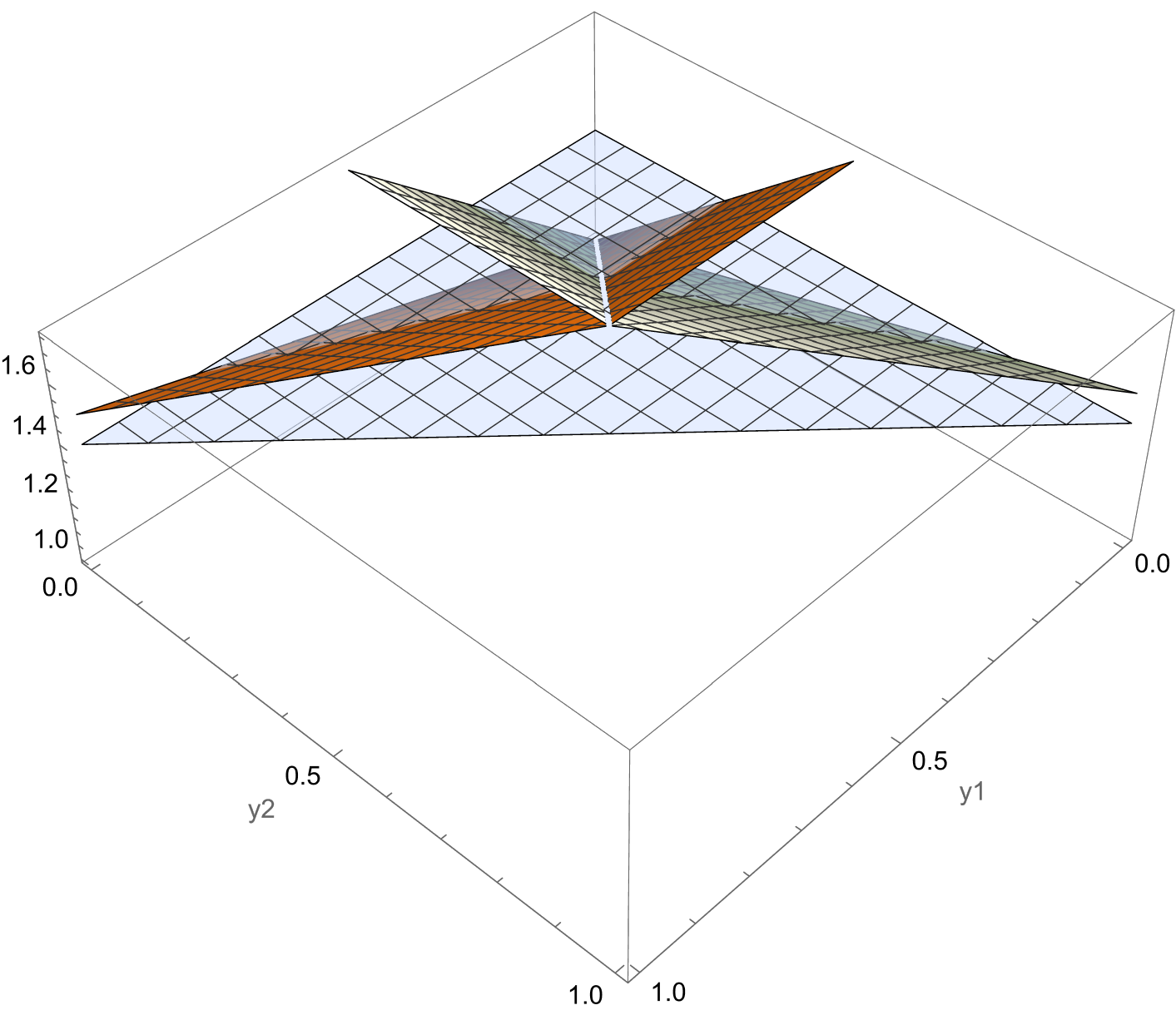}
\caption{Boundary functions for {Player 1 (green) and Player 2 (orange), shown from two different perspectives. Here} $k=c=\rho=\mu=\theta=1$, $\beta=1/2$. The waiting and installation regions are, respectively, those below and above the surfaces. The blue horizontal plane highlights a level set of the boundary functions (cf. Figure \ref{fig:onestep_1}).} \label{fig:onestep_4}
\end{figure}
 
Finally, let us {represent the Nash equilibria described in cases i) through iv) above, and their respective value function, in terms of} the {boundary functions $F_1, F_2$} and the regions in \eqref{eq:onestep_wait}. First, let 
\begin{align}\label{eq:Ri}
R_i(x,y):=S_i(x,y,0,0)= y_{i} \frac{x\rho+\mu k - \b k \langle {\bf 1},y\rangle }{\rho (\rho+k)}, \qquad (x,y)\in \R\times \overline{D}, \qquad i=1,2,
\end{align}
be the expected profit associated to a non-installation strategy. Then it is clear that 
\begin{equation}\label{eq:profit_onestep2}
S_i(x,y,I)=R_i(x,y+I).
\end{equation}
Next, we set ${F}^{-1}:[F(0,0),\infty)\to {[0,\theta/2]}$ as
\begin{equation}\label{eq:invF}
{F}^{-1}(x):= \begin{cases}
\frac{2}{3}A(x) &\quad \text{if}\quad x\in [{F}(0,0),{F}(\theta/2,\theta/2)],\\ 
\frac{\theta}{2} & \quad \text{if}\quad x> F(\theta/2,\theta/2 ).
\end{cases}
\end{equation}
{In other words, ${F}^{-1}$ is the inverse of the function $[\eta\mapsto F(\eta,\eta)]$ (which is injective in light of Remark \ref{rem:increasing_F}) on $[F(0,0),F(\theta/2,\theta/2)]$, and it is constantly equal to $\theta/2$ on $(F(\theta/2,\theta/2),+\infty)$.}
For any $\xi\in [0,\theta]$, also set
${F}_{\xi}^{-1}:[F(\xi,0),\infty)\to [0,(\theta-\xi)\wedge \xi]$ as
\begin{equation}\label{eq:invFi}
{F}_{\xi}^{-1}(x):= \begin{cases}
A(x)-\frac{\xi}{2} &\quad \text{if}\quad x\in [{F}(\xi,0),{F}(\xi,{(\theta-\xi)\wedge \xi})],\\ 
(\theta-\xi)\wedge \xi  & \quad \text{if}\quad x> {F}({\xi,(\theta-\xi)\wedge \xi})  .
\end{cases}
\end{equation}
In other words, ${F}_{\xi}^{-1}$ is the inverse of the function $\big[\eta\mapsto F(\xi,\eta) =F_2(\xi,\eta) =F_1(\eta,\xi)\big]$ (which is injective in light of Remark \ref{rem:increasing_F}) on $[F(\xi,0),F({\xi,(\theta-\xi)\wedge \xi})]$, and it is constantly equal to ${(\theta-\xi)\wedge \xi}$ on $(F({\xi,(\theta-\xi)\wedge \xi}),+\infty)$.

With these notations at hand, we can characterize the equilibrium strategy as follows: 
if {Player $1$} is in its installation region and {Player $2$} is not, then Player $1$ makes an installation of {size $I^1={F}^{-1}_{y_2}(x)-y_1$}, pushing the initial state of the system to the boundary of $\mathbb{W}^1$ (therefore of $\mathbb{W}^1\cap \mathbb{W}^2$), constant in the direction $(x, \, y_2)$. {The strategies are symmetric when Player 2 is in its installation region and Player 1 is not}. If {Players $1$ and $2$} are both in their installation region, they both make an installation $I={F}^{-1}(x)-y$, again pushing the state on the boundary of {$\mathbb{W}^1\cap \mathbb{W}^2$. Note that both players reach the level ${F}^{-1}(x)$ after a joint installation.} {To sum up:
\begin{equation}\label{eq:eq_strat_one_s}
I=\begin{cases}
(0,0), & (x,y)\in \mathbb{W}^1\cap \mathbb{W}^2,\\
(0, {F}^{-1}_{y_1}(x)-y_2), & (x,y)\in \mathbb{W}^1\cap \mathbb{I}^2, \\
({F}^{-1}_{y_2}(x)-y_1 , 0), & (x,y)\in \mathbb{I}^1\cap \mathbb{W}^2, \\
({F}^{-1}(x)-y_1, {F}^{-1}(x)-y_2), & (x,y)\in \mathbb{I}^1\cap \mathbb{I}^2.
\end{cases}
\end{equation}
}
Therefore, recalling Equation \eqref{eq:profit_onestep2}, {the value function of Player 1 is 
\begin{align}
V_1(x,y) = \mathcal{S}_1(x,y,I)=\begin{cases}
R_1(x,y), & (x,y)\in \mathbb{W}^1\cap \mathbb{W}^2,\\
R_1(x,y_1, {F}_{y_1}^{-1}(x) ), & (x,y)\in \mathbb{W}^1\cap \mathbb{I}^2, \\
R_1(x,{F}_{y_2}^{-1}(x) , y_2) - c\, ({F}_{y_2}^{-1}(x) - y_1), & (x,y)\in \mathbb{I}^1\cap \mathbb{W}^2, \\
R_1(x,{F}^{-1}(x),{F}^{-1}(x))-c\,({F}^{-1}(x)-y_1), & (x,y)\in \mathbb{I}^1\cap \mathbb{I}^2.
\end{cases}\label{eq:value_onestep}
\end{align}
{By symmetry, the value function for Player 2 is $V_2(x,y_1,y_2) =V_1(x,y_2,y_1) $.}

\subsection{Comparison with Pareto optima} 
The Pareto optimum is an installation strategy which maximizes the aggregate expected profit 
$$\sum_{i=1,2}\mathcal{S}_i(x,y_1,y_2,I_1,I_2)$$
{over all the $I=(I^1,I^2)\in \R^2$ satisfying the constraints \eqref{eq:static_admissible}.}
Again, this quantity can be computed exactly: setting $\bar y=y_1+y_2$ and $\bar I=I_1+I_2$ we have 
\begin{align*}
\sum_{i=1,2}\mathcal{S}_i(x,y_1,y_2,I_1,I_2)
&=(\bar y +\bar I){\Eb}\left[\int_0^\infty e^{-\rho \tau}{X}^{x,y,I}(\tau)d\tau\right]-c\bar I\\
&=(\bar y +\bar I)\frac{x\rho+\mu k-\beta k (\bar y+\bar I)}{\rho (\rho+k)}-c\bar I.
\end{align*}
Optimizing w.r.t the cumulative installation $\bar I {\in [0,\theta - \bar y]}$ we obtain 
\begin{equation}
\bar I=\min\{A(x)-\bar y, \theta-\bar y\}\vee 0, 
\end{equation}
where $A(x)$ is as in \eqref{A(x)}. Clearly, depending on the proportion in which the two players install, there are infinitely many Pareto optima. 
However, even in the case $y_1=y_2$ and assuming $I_1=I_2$, the Pareto optimum is fundamentally different from the Nash equilibrium: for instance it is easy to see that for the former, the system reaches saturation for $A(x)\geq \theta$, while for the latter the system reaches saturation earlier, when $A(x)\geq \frac{3}{4}\theta$ {(cf. previous section)}. 

\begin{figure}[!h]\centering
\includegraphics[scale=0.4]{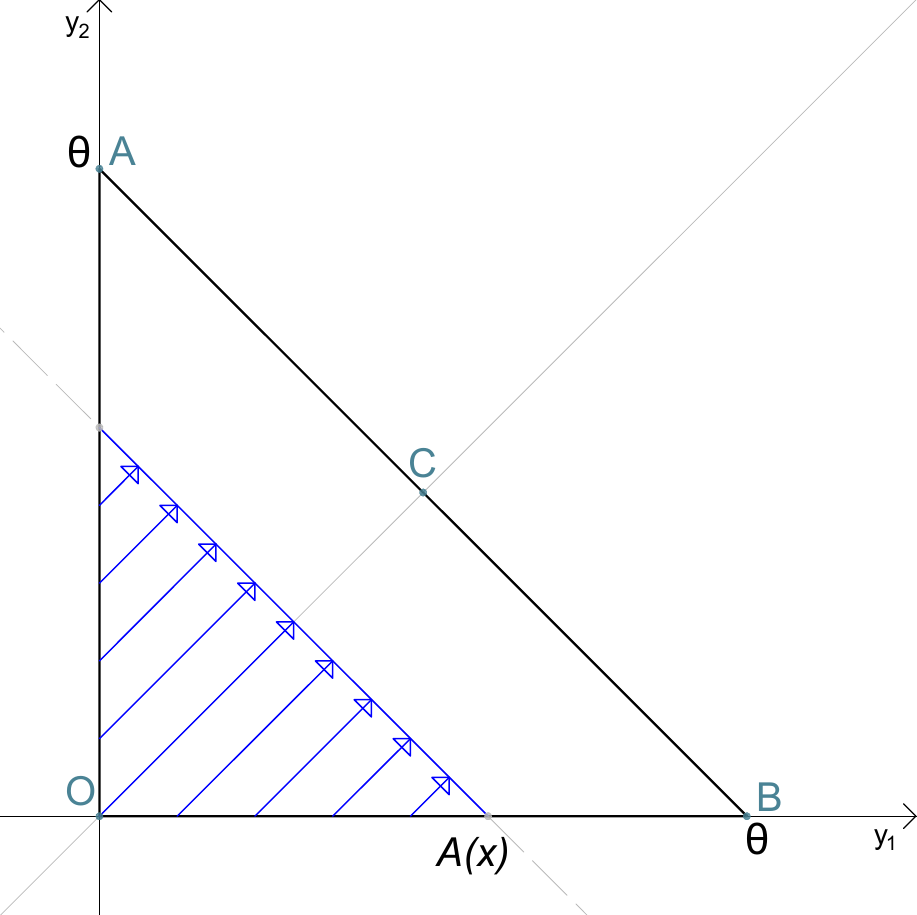} $\quad$
\includegraphics[scale=0.44]{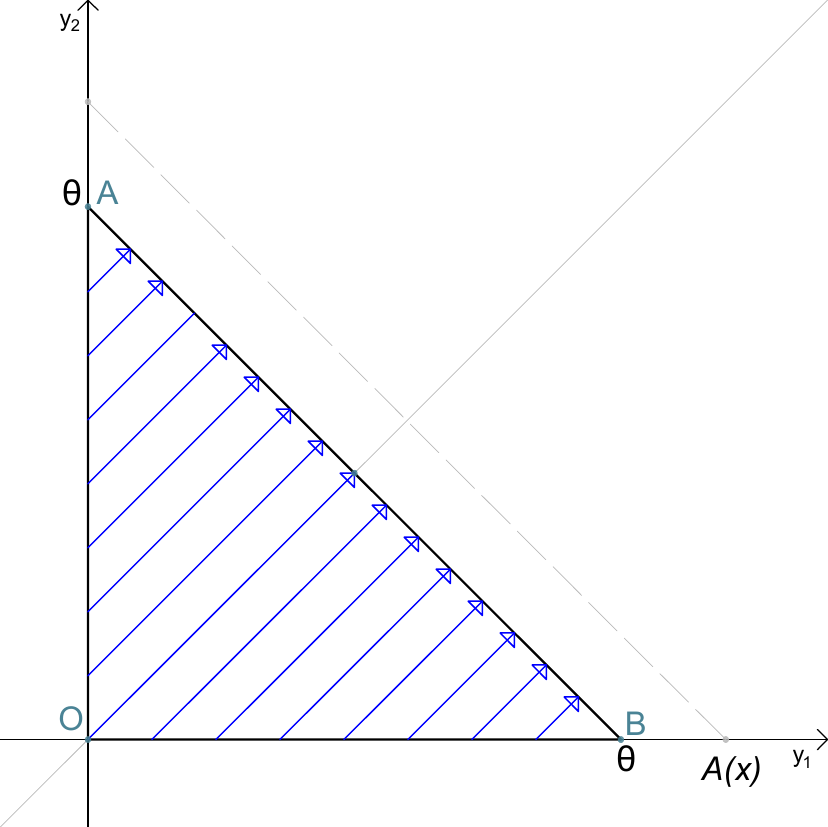}
\caption{A possible Pareto optimum{: without saturation (left) and with saturation (right).}} 
\end{figure}

\subsection{Final remarks and ansatz for the general game}\label{sec:remarks}
Similarly to the one step case, a Nash equilibrium for the general game will keep the state process $(X^{x,y,I}_t,Y^{y,I}_t)$ in the closure of $\mathbb{W}^1\cap \mathbb{W}^2$ by a possible lump installation at time $0$, and successive infinitesimal installations for $t>0$. With this in mind, the previous computations provide a good intuition on the structure of the regions $\mathbb{W}^i$ and $\mathbb{I}^i$, $i=1,2$. On one hand, the cross structure of the boundaries, as well as their monotonicity can be heuristically explained by the principle that the firm which has the least installed power should install before the other. Then, it would be also true that the joint installation region should be a square for every fixed $x$, as seen in Figure \ref{fig:onestep_1}. Indeed, if that was not the case, then there would be states of the system which are transported to an installation region by any admissible control (see Figure \ref{fig:onestep_5}){, which is a contradiction in terms.} This provides a good motivation for assuming the \emph{$F$-structure} for the waiting and installation regions, which arises explicitly in the static game (cf. Section \ref{sec:free_bound_static}), to be in place also in the dynamic game.

\begin{figure}[!h]
\centering
\includegraphics[scale=0.5]{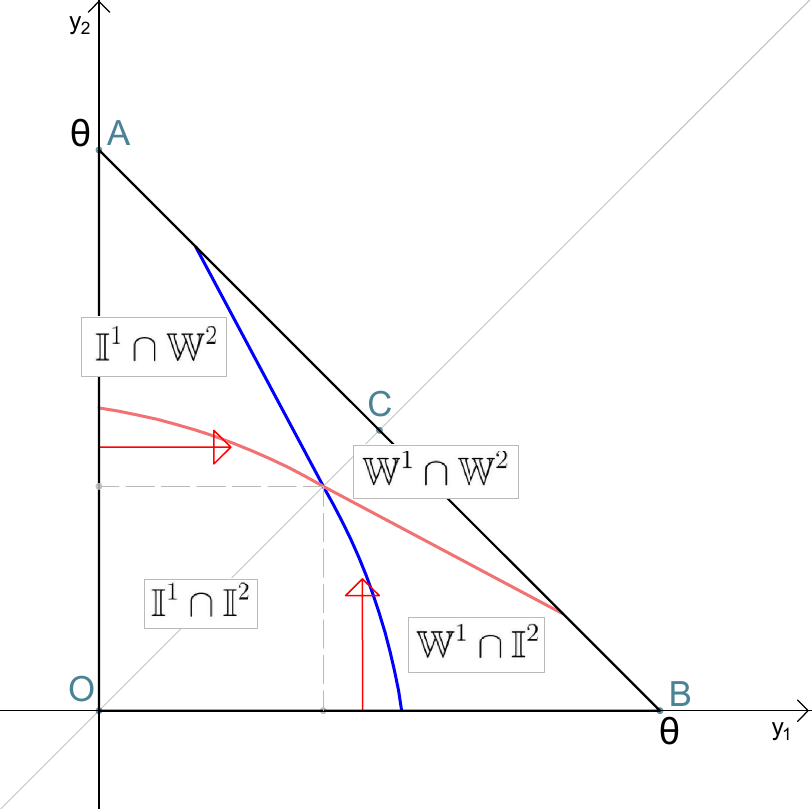}$\quad$
\includegraphics[scale=0.38]{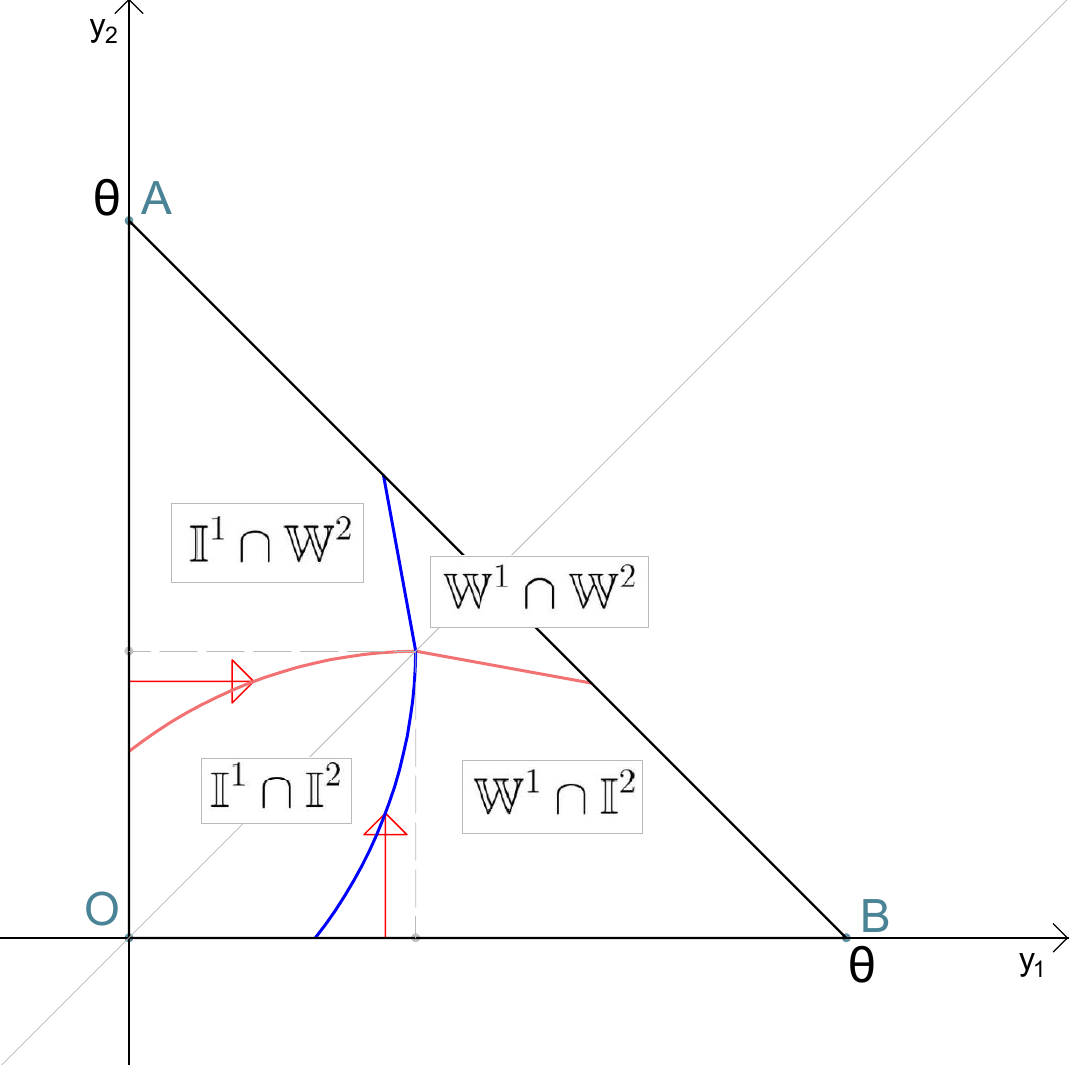}
\caption{Two non-admissible boundaries: {a joint installation leads to an individual installation region (left), and an individual installation leads to a joint installation region (right). }}\label{fig:onestep_5}
\end{figure}

\newpage

\section{Continuous game}

\subsection{HJB free-boundary system and verification theorem}\label{sec:ver}

Let us now go back to the original {dynamic game introduced in Section \ref{sec:game}}. In this section we prove a verification theorem, which characterizes the equilibrium solution of the game. 

{Denoting by ${\bf 0}$ the admissible control $I_t \equiv (0,0)$, let}
$\cL$ be the infinitesimal generator of the diffusion $(X_t^{x,y,{\bf 0}})_{t\geq 0}$, which is given by the second order operator
\begin{equation}\label{eq:infinitesimal_generator}
\cL f=\cL f(x,y):=\frac{1}{2} \sigma^2 \p_{xx}f(x,y)+k(\mu - \beta \langle {\bf 1},y\rangle -x)\p_xf(x,y),
\end{equation}
for $f(\cdot,y)\in C^2$, $y\in\overline{D}$, and let $R_i(x,y)$ be the function defined in \eqref{eq:Ri}, which corresponds to the expected profit of firm $i$, $i=1,2$,  when both firms follow a non-installation strategy, namely
\begin{equation}
R_i(x,y) = S_i(x,y,{\bf 0}), \qquad (x,y)\in \R\times \overline{D}, \qquad i=1,2.
\end{equation}
\paragraph{Heuristics.} Let $\{i,j\}$ be a permutation of $\{1,2\}$ and assume that Player $j$ installs optimally. Following the typical structure of singular control optima, the pair $(X_t, Y_t)$ is expected to lie inside the closure of a waiting domain $\mathbb{W}^j \subset \R \times D$, at any positive time. In particular, being $\mathbb{W}^j$ the domain where Player $j$ is idle, the differential of $Y^j_t$ is expected to be zero whenever $(X_t, Y_t) \in \mathbb{W}^j$. Thus, following the same arguments as in \cite{MR4305783}, the value function of Player $i$ must satisfy, on $\mathbb{W}^j$, the same variational inequality as in the single-player case, namely
\begin{equation}\label{HJB_bis}
\max\{(\cL-\rho)V_i(x,y)+xy_i,\, \p_{y_i}V_i(x,y)-c\}=0, \qquad(x,y)\in \mathbb{W}^j. 
\end{equation}
On the other hand, when $(X_t, Y_t) \notin \mathbb{W}^j$, one should expect that the application of It\^o formula to $V_i (X_t, Y_t)$ will yield an a-priori non negligible contribution of type $\partial_{y_j} V_i (X_t, Y_t) dI^{j}_t$. In order to prove a verification theorem, this quantity must be killed by the condition
\begin{equation}\label{HJB_ter}
\p_{y_j}V_i(x,y) = 0, \qquad (x,y)\in \mathbb{I}^j:= (\R\times D) \setminus \mathbb{W}^{j}.
\end{equation}
Finally, the following boundary condition should clearly hold true:
\begin{equation}\label{eq:terminal_cond}
V_i(x,y)=R_i(x,y), \qquad {x\in\R, \quad y\in \overline{D} \text{ s.t. } y_1 + y_2 = \theta.} 
\end{equation} 
Note that these equations must be satisfied for both the choices $(i,j) = (1,2)$ and $(i,j) = (2,1)$. 

\bigskip

Unfortunately, the chances of proving a verification theorem for the system given by \eqref{HJB_bis}-\eqref{HJB_ter}-\eqref{eq:terminal_cond} seem scarce unless one requires some additional structure for the solutions. Motivated by the remarks in Section \ref{sec:remarks}, we then impose that the regions $\mathbb{W}^i$ and $\mathbb{I}^i$, $i=1,2$, are given according to the \emph{$F$-structure} that we explicitely described in the one-step game at Section \ref{sec:free_bound_static}. 
\begin{definition}[Admissible boundary]\label{def:admissible_boundary}
A real-valued function ${F=F(y_1,y_2)}$ defined on {$\overline{D}\cap \{ y_1\leq y_2 \}$ (or on $\overline{D}\cap \{ y_2\leq y_1 \}$)} is called a {\emph{$1$-admissible (or $2$-admissible) boundary}} if it is continuous and strictly increasing in both single variables.
\end{definition}
We now define the waiting and installation regions associated to a $j$-admissible boundary:
\begin{notation}\label{not:barriers}
Let $\{i,j\}$ be a permutation of $\{1,2\}$. For any $j$-admissible boundary $F$, we set
\begin{align}
\p F&:=\{(x,y)\in\R\times {\overline{D}} : y_i\geq y_j, x = F(y)  \},
\intertext{which corresponds to the graph of the function $F$. Moreover we set}
W^{\text{free}}_F&:=\{ (x,y)\in\R\times D : y_i\geq y_j, x < F(y)  \}, \\
W^{\text{prol}}_F&:=\{ (x,y)\in\R\times D : y_i\leq y_j, y_j < \theta/2, x < F(y_j,y_j)  \},\\
W^{\text{sat}}_F&:=\{ (x,y)\in\R\times D :  y_j \geq \theta/2 \},\\
W_F &:= W^{\text{free}}_F \cup W^{\text{prol}}_F \cup W^{\text{sat}}_F , \\
I_F &:= (\R\times D) \setminus W_F .
\end{align} 
\end{notation}
Here the interpretation of $W^{\text{free}}_F$, $W^{\text{prol}}_F$ and $W^{\text{sat}}_F$ is analogous to that in Remark \ref{rem:interpretation_W}.
For a given permutation $\{i,j\}$ of the indices $\{1,2\}$, we now introduce the notion of regular solution to the variational problem associated to a $j$-admissible boundary $F$, namely 
\begin{equation}\label{HJB_quat}
\begin{cases}
\max\{(\cL-\rho)V(x,y)+xy_i,\, \p_{y_i}V(x,y)-c\}=0, &\quad(x,y)\in W_F \\
\p_{y_j}V(x,y) = 0, &\quad (x,y)\in I_F
\end{cases}, 
\end{equation}
together with the boundary condition  
\begin{equation}\label{eq:terminal_cond_bis}
V(x,y)=R_i(x,y), \qquad {x\in\R, \quad y\in \overline{D} \text{ s.t. } y_1 + y_2 = \theta.} 
\end{equation}
The idea behind the following definition is that the variational inequality on $W_F$ splits the latter set in two regions, divided by a free boundary that is determined by an $i$-admissible boundary $G$.

\begin{definition}[Regular solution to the variational problem]\label{def:regular_sol_VP}
Let $\{i,j\}$ be a permutation of $\{1,2\}$ and $F$ be a $j$-admissible boundary. A pair $(V,G)$ is a \emph{regular solution to  $\text{VP}_{\{i,j\}}(F)$} if $G$ is an $i$-admissible boundary and $V:\R \times \overline{D} \to \R$ is a continuous function satisfying:
\begin{itemize}
\item[(i)] [Regularity w.r.t. $x$] $\p_{x}V \in C(\R\times D)$ and $\p_{xx}V\in C(\overline{W_F})$. Furthermore,
for any $R>0$, there exists $L_R>0$ such that
\begin{equation}\label{eq:Lip_x}
\left|(\p_xV)(x,y)-(\p_xV)(x',y)\right| \le L_R |x-x'|,\qquad x,x'\in [-R,R], \quad y\in {{D}}.
\end{equation}
 \item[(ii)] [Regularity w.r.t. $y$] $\p_{y_i} V \in C(\overline{W_F})$ {and $\p_{y_j} V$ is continuous on a lower open neighbourhood of ${I_{F}}$. 
More precisely, there exists $r>0$ such that $\p_{y_j} V\in C(U_r)$, with $U_r := \{(x,y) \in \R^3 : (x+r,y)\in I_F\}$}.
\item[(iii)] [Solution to HJB equation \eqref{HJB_quat}] The following relations are satisfied:
\begin{align}\label{eq:regular_sol}
(\cL-\rho)V(x,y)+x y_i =0,\ \  \p_{y_i}V(x,y)-c \leq 0,&\qquad (x,y)\in W_F \cap W_G 
 ,\\ \label{eq:regular_sol_bis}
(\cL-\rho)V(x,y)+x y_i \leq 0, \ \ \p_{y_i}V(x,y)-c = 0,&\qquad (x,y)\in  W_F \cap I_G, \\
\p_{y_j}V(x,y) = 0,&\qquad (x,y)\in  I_F. \label{eq:regular_sol_ter}
\end{align}
In particular, \eqref{HJB_quat} is satisfied pointwise. 
\item[(iv)][(Boundary condition)] Condition \eqref{eq:terminal_cond_bis} is satisfied. It follows, from a direct computation, that $\p_{xx}V\in  C (\R \times \{ y_1 + y_2 = \theta \}\cap \overline{D})$ and, for any $y\in \overline{D}$ with $y_1 + y_2 = \theta$, we have
\begin{equation}
(\cL-\rho)V(x,y)+x y_1 =0, \qquad x\in\R.
\end{equation}
\item[(v)] [Sublinear growth] There exists a constant $K>0$ such that
\begin{equation}\label{eq:linear_growth_V}
|V(x,y)| \leq K (1+|x|), \qquad (x,y)\in \R\times \overline{D}.
\end{equation}
\end{itemize}
We call $G$ a \emph{free-boundary function for $\text{VP}_{\{i,j\}}(F)$}.
\end{definition}

\begin{remark}
In the single player control setting \cite{MR4305783} the second order derivative with respect to the $x$-variable is continuous everywhere {on $\R$}. However, it is common to observe a loss in regularity when passing from a control setting to the associated games involving free boundaries: to this extent, see e.g.  \cite{chen2013nonzero}, which finds a solution for an optimal stopping game by assuming the value functions to be only continuous at the free boundary (instead of the $C^1$ smoothness of the classical verification theorem contained in 
), or the discussion in \cite{MR4066995} highlighting exactly the same phenomenon in the context of impulsive games. Similarly, in optimal {stopping problems} it is generally difficult to prove that the value function is regular up to the free boundary, see the discussion and examples in \cite{MR4620174}.
\end{remark}

We are now in the position to give a rigorous interpretation of the the variational system \eqref{HJB_bis}-\eqref{HJB_ter}-\eqref{eq:terminal_cond} obtained before with heuristic arguments. The main idea is that the free-boundary $G$ in the regular solution to the  $\text{VP}_{\{i,j\}}(F)$ should coincide with the reflection (see definition below) of $F$ itself.

Precisely, we introduce two reflection operators. Let ${\bf r}: \R^2 \to \R^2$ be given by
\begin{equation}
{\bf r}(y_1,y_2) = (y_2,y_1),
\end{equation} 
and its extension $\bar{\bf r}: \R \times \R^2 \to  \R \times \R^2$ be given by 
\begin{equation}
\bar{\bf r} (x,y) =  \big(x,{\bf r}(y)\big).
\end{equation}
Let also 
$\cR: C(D) \to C(D)$ and $\bar\cR: C(\R\times D) \to C(\R\times D)$ be acting, respectively, as 
\begin{equation}
\cR(f) = f \circ {\bf r},\qquad \bar\cR(f) = f \circ \bar{\bf r}.
\end{equation}
\begin{remark}
Let $\{i,j\}$ be a permutation of $\{1,2\}$. 
If $F$ is a $j$-admissible boundary, then $\cR(F)$ is an $i$-admissible boundary and 
\begin{equation}
W^{\text{free}}_{\cR(F)} = \bar{\bf r}(W^{\text{free}}_{F}), \quad W^{\text{prol}}_{\cR(F)} = \bar {\bf r}(W^{\text{prol}}_{F}), \quad W^{\text{sat}}_{\cR(F)} = \bar {\bf r}(W^{\text{sat}}_{F}), \quad W_{\cR(F)} = \bar {\bf r}(W_{F}),\quad I_{\cR(F)} = \bar {\bf r}(I_{F}). 
\end{equation}
\end{remark}

\begin{definition}[Equilibrium solution to the variational problem]\label{def:equilibrium_sol}
We say that a pair $(V,F)$ is an {\em equilibrium solution} to VP if $F$ is a $2$-admissible boundary and $(V,\cR (F))$ is a regular solution to the variational problem $\text{VP}_{\{1,2\}}(F)$.
\end{definition}
\begin{remark}[Symmetry]\label{rem:equilibrium_sol}
It is straightforward to see that $(V,F)$ is an equilibrium solution to VP if and only if $\cR (F)$ is a $1$-admissible boundary and $(\bar\cR (V),F)$ is a regular solution to the variational problem $\text{VP}_{\{2,1\}}(F)$.
\end{remark}

\begin{notation}\label{notation_V_F}
If $(V,F)$ is a fixed equilibrium solution to VP, we set
\begin{align}
&F_2:= F, \quad V_1:= V,&& F_1:= \cR (F), \quad V_2:= \bar\cR(V),\\
&\Wb^2:= W_{F_2} , \quad \Ib^2:= I_{F_2}, && \Wb^1:= W_{F_1}, \quad \Ib^1:= I_{F_1}.
\end{align}
\end{notation}
\begin{definition}[$j$-optimal control]\label{def:i_optimal_control}
Let $j\in\{ 1,2 \}$, $F$ a $j$-admissible boundary and $(x,y)\in \R\times {D}$. A control $I\in \cI(y)$ is called \emph{$j$-optimal with respect to $F$ and $(x,y)$} if:
\begin{itemize}
\item[(i)] both the following conditions hold for {any $t\geq 0$}: 
\begin{align}
&{\Pb\Big( (X^{x,y,I}_t,Y^{y,I}_t) \notin \overline{W_F} , Y^{y,I}_t \in D  \Big) = 0 },
\label{eq:cond_a_i_optimal}\\
&(X^{x,y,I}_t, Y^{y,I}_t) \in W_F \Rightarrow {d Y^{y,I,j}_{t+}} = 0.
\label{eq:cond_b_i_optimal} 
\end{align}
\item[(ii)] $I^j$ is {$P$-a.s.} continuous {on $[0,+\infty)$}, with
\begin{equation}\label{eq:cond_c_i_optimal}
I^j_0 = \min\{ s\geq 0 : (x,y+s \, e_j)  \in \overline{W_F} \},
\end{equation}
with a possible jump from $0^-$ to $0$.
\end{itemize}
\end{definition}
\begin{remark}In other words, given a $j$-admissible boundary $F$, a $j$-optimal control keeps the state $(X^{x,y,I}_t,Y^{y,I}_t)$ inside the closure of the waiting region of player $j$, $\overline{W_F}$, with probability $1$ for any $t\geq 0$ with the least effort. In particular, if $(X^{x,y,I}_{0^-},Y^{y,I}_{0^-})\notin \overline{W_F}$, then $\Delta I^j_{0^-}=I^j_0-I^j_{0^-}>0$ so that it can be (only) discontinuous at time $0^-$. On the other hand, if the state is in $W_F$ at some $t\in [0^-,\infty)$, then $I^j_t=0$.
\end{remark}
The following Verification theorem states that a Nash equilibrium is a strategy associated to a control whose components are optimal w.r.t. the boundaries of both players, identified by an equilibrium solution to VP. Thus it will keep the state inside the intersection of the waiting regions of both players, for any $t\geq 0$.

\begin{theorem}[Verification theorem]\label{th:verific_th}
Let $(V,F)$ be an equilibrium solution to VP. Let also $\contr^{\star}\in\mathcal{A}$ such that, for any $j=1,2$ 
 and $(\bar{x},\bar{y})\in \R\times {D}$, every control process $I\in\mathcal{I}(y)$ such that 
\begin{equation}\label{eq:cond_Iit}
I^j_t := \contr_t^{\star,j}\big((X^{\bar x,\bar y,I}_s)_{s\leq t}, (Y^{\bar y,I}_{s} )_{s\leq t}\big), \qquad t\geq 0, 
\end{equation}
is $j$-optimal w.r.t. ${F^j}$ and $(\bar{x},\bar{y})$.

Then $\contr^{\star}$ is a Nash equilibrium and 
\begin{equation}\label{eq:nash_cond_th}
\cS_i(\bar x,\bar y; I^{\star}) = V_i(\bar x,\bar y) , \qquad (\bar x,\bar y)\in \R\times D,\quad  i = 1,2,
\end{equation}
where $I^{\star}\in\mathcal{I}(\bar y)$ is the unique admissible control such that 
\begin{equation}\label{eq:I_markov_star}
I^{\star}_t = \contr_t^{\star}\big((X^{\bar x,\bar y,I^{\star}}_s)_{s\leq t}, (Y^{\bar y,I^{\star}}_{s})_{s\leq t}\big), \qquad t\geq 0, 
\end{equation}
in agreement with Definition \ref{def:markov_strategy}.
\end{theorem}
To prove Theorem \ref{th:verific_th} we need some preliminary results. Lemma \ref{lem:boundary} below guarantees that the pair $(X^{x,y,I}_t,Y^{y,I}_t)$ takes values, almost surely, on a set where the value function is smooth. This would be clear if the law of $(X^{x,y,I}_t,Y^{y,I}_t)$ were absolutely continuous with respect to the Lebesgue measure; however the marginal law of $Y^{y,I}_t$ is not, a priori, absolutely continuous, and no information is available for the joint law of the process. What we do exploit is the monotonicity of $F$ (see Definition \ref{def:admissible_boundary}) and of $Y^{y,I}$ (see Eq. \eqref{eq:def_strat_ammiss}) and the fact that $I$ is optimal with respect to $F$.
This allows us to prove an {\em ad hoc} It\^o formula, accounting for the lack of regularity in the proof of the verification theorem. 
\begin{lemma}\label{lem:boundary}
{Let $\{i,j\}$ be a permutation of $\{ 1,2 \}$, let $F$ be a $j$-admissible boundary, and $(\bar x, \bar y)\in \R\times {D}$.} 
If $I\in \cI(\bar y)$ is {$j$-optimal with respect to $F$ and $(\bar x, \bar y)$}, then we have
\begin{equation}\label{eq:boundary}
\Pb\left((X^{x,\bar y,I}_t,Y^{\bar y,I}_t)\in \p W_F , Y^{\bar y,I}_t \in D \cap \{y_2 < \theta/2 \} \right)=0, \qquad a.e. \quad t\geq 0.
\end{equation}
\end{lemma}
\proof Let us assume $\{i,j\} = \{1,2\}$ without losing generality. To ease notation, we also drop the superscripts in $X^{\bar x,\bar y,I}_t$ and $Y^{\bar y,I}_t$. Setting
\begin{equation}
A_t := \big\{{(X_t,Y_t)\in  \p W_F, Y_t\in   D \cap \{y_2 < \theta/2 \} }\big\} ,
\end{equation}
notice that
\begin{equation}
A_t = \left\{{X_t= F(Y_t) , Y^2_t \leq Y^1_t, Y^{y,I}_t\in D }\right\} \cup  \left\{{X_t= F(Y^2_t, Y^2_t) , Y^1_t < Y^2_t < \theta/2 }\right\} .
\end{equation}
Thus we have
\begin{equation}\label{eq:inclusion}
A_t 
\subset \{ X_t = \sup_{s\in[0,t]} X_s \}.
\end{equation}
To show this, assume that $X_s > X_t$ for some $s<t$. Then, as $F$ is increasing in both variables and $I$ is $2$-optimal with respect to $F$ and $(\bar x, \bar y)$, $A_t$ implies that either $Y^{1}_s> Y^{1}_t$ or $Y^{2}_s> Y^{2}_t$, which violates $I\in \cI(\bar y)$.

Also, by Girsanov's theorem (see \cite[Theorem 10.5]{MR2791231}), there exists $\tilde\Pb \sim \Pb$ such that $X$ is a Brownian motion with respect to $\tilde\Pb$. In particular, we have $\tilde\Pb( X_t = \sup_{s\in[0,t]} X_s) = 0$ (see for instance the classical reference \cite[Chapter 2.8]{Karatzas}). This and \eqref{eq:inclusion} yield $\Pb(A_t)=0$. 
\endproof

\begin{lemma}[Weak It\^o formula]\label{lem:ito}
Let $\{i,j\}$ be a permutation of $\{ 1,2 \}$. Let $F$ be a $j$-admissible boundary and $V:\R\times \overline{D} \to \R$ a continuous function satisfying the regularity assumptions (i)-(ii)-{(iv)} in Definition \ref{def:regular_sol_VP}. Then, for any $(\bar x, \bar y)\in \R\times \overline{D}$, and any $I\in \cI(\bar y)$ {$j$-optimal with respect to $F$ and $(\bar x, \bar y)$}, we have
\begin{align}
 V(X^{\bar x,\bar y,I}_t, Y^{\bar y,I}_t) & = V({\bar x,\bar y}) + \int_{0}^t \cL V (X^{\bar{x},\bar{y},I}_s , Y^{\bar{y},I}_s ) ds + \sigma \int_0^t \partial_x V (X^{\bar{x},\bar{y},I}_s , Y^{\bar{y},I}_s )\, d W_s \\
& \quad + \int_0^t \big\langle \nabla_y V (X^{\bar{x},\bar{y},I}_s , Y^{\bar{y},I}_s ) , d I^c_s \big\rangle + \sum_{
{0\le s\le t}} \big( V (X^{\bar{x},\bar{y},I}_s , Y^{\bar{y},I}_s ) - V (X^{\bar{x},\bar{y},I}_s , Y^{\bar{y},I}_{s-} )  \big) , \qquad t\geq 0. \\
\label{eq:ito}
\end{align}
Here the equality is meant $\Pb$-almost surely, and the process $I^c$ denotes the continuous part of $I$.
\end{lemma}
\begin{remark}\label{rem:ito}
Note that the It\^o formula \eqref{eq:ito} is non-trivial only when $(X_t^{\bar x,\bar y,I},Y_t^{\bar y,I})\in D$ and $Y_t^{\bar y,I,j}<\theta/2$. Precisely, 
setting the stopping time
\begin{equation}\label{eq:tausat}
\tau_{j,\text{sat}}:=\begin{cases} 
\inf J_{j,\text{sat}}& \text{if } J_{j,\text{sat}} \neq \emptyset  \\
+\infty & \text{otherwise}
\end{cases}, \qquad \mbox{ with } \qquad J_{j,\text{sat}}:=  \{ t\geq0 : Y^{\bar y, I ,j}_t \geq \theta/2\}  ,
\end{equation}
we have $(X_t^{\bar x,\bar y,I},Y_t^{\bar y,I})\in  W_F^{\text{sat}}$ and $Y^{\bar y,I,j}_t = Y^{\bar y,I,j}_{\tau_{j,\text{sat}}}$ for any $t\geq \tau_{j,\text{sat}}$, and thus \eqref{eq:ito} reduces to the standard It\^o formula as $\partial_{xx}V$ and $\partial_{y_i}V$ are continuous on ${W_F^{\text{sat}}}$. Analogously, setting the stopping time
\begin{equation}\label{eq:tausat_bis}
\tau_{\text{sat}}:=\begin{cases} 
\inf J_{\text{sat}}& \text{if } J_{\text{sat}} \neq \emptyset  \\
+\infty & \text{otherwise}
\end{cases}, \qquad  \mbox{ with } \qquad J_{\text{sat}}:=  \{ t\geq 0 : Y^{\bar y, I ,1}_t+Y^{\bar y, I ,2}_t = \theta \}  ,
\end{equation}
we have that $Y_t^{\bar y,I,1} + Y_t^{\bar y,I,2} = \theta$ with $Y^{\bar y,I}_t = Y^{\bar y,I}_{\tau_{\text{sat}}}$ for any $t\geq \tau_{\text{sat}}$, and thus \eqref{eq:ito} reduces again to the classical It\^o formula as $\partial_{x}V$ is Lipschitz continuous with respect to the variable $x$. 

On the other hand, for $t\in [0,\tau_{j,\text{sat}}\wedge\tau_{\text{sat}}[$, the regularity conditions (i)-(ii) in Definition \ref{def:regular_sol_VP} are sufficient to write \eqref{eq:ito}. Indeed, the $j$-optimality of $I$ guarantees that the process $(X^{\bar x,\bar y,I}_t, Y^{\bar y,I}_t)\in \overline{W_F}$, and the partial derivatives $\p_{y_i}V$, $\p_{xx}V$ are continuous on $\overline{W_F}$. Moreover, if $(X^{\bar x,\bar y,I}_t, Y^{\bar y,I}_t)\in W_F$, then $dY^{x,I,j}_{t^+}=0$ by \eqref{eq:cond_b_i_optimal}, so that $\p_{y_j}V$ only needs to be well defined in a lower neighbourhood of $I_F$.
\end{remark}
\begin{proof}
Let us fix $(\bar x, \bar y)\in {\R \times \overline{D}}$, and $I\in \cI(y)$ {$j$-optimal with respect to $F$ and $(\bar x, \bar y)$}. 
To ease notation, we will drop the superscripts in the processes $X^{\bar{x},\bar{y},I}$ and $Y^{\bar{y},I}$, {and fix $\{i,j\}=\{1,2\}$} throughout the proof. 

It is not restrictive to assume $(\bar x, \bar y) \in \overline{W_{F}}$. Indeed, if $(\bar x, \bar y) {\notin \overline{W_{F}}}$, 
we have
\begin{equation}
V(X_t, Y_t) - V(\bar x , \bar y ) = V(X_t, Y_t) - V(X_{0}, Y_{0}) + V(X_{0}, Y_{0}) - V(\bar x , \bar y),
\end{equation}
where $(X_{0}, Y_{0})=(\bar x, \bar y+I_0)\in \overline{W_{F}}$ by Definition \ref{def:i_optimal_control}. {Furthermore, in light of Remark \ref{rem:ito}, it is not restrictive to assume $\bar{y}\in D$ with $\bar y_2 < \theta/2$, and it is enough to prove \eqref{eq:ito} up to the first (and last) exit time of $Y$ from $D\cap\{ y_2<\theta/2 \}$. Namely, setting
\begin{align}
S^V_t &:=  V(X_{t}, Y_{t}) -  V({\bar x,\bar y}) -  \sum_{
{0\le s\le {t}}} \big( V (X_s , Y_s ) - V (X_s , Y_{s-} )  \big), \\
R^V_t &: =  \int_{0}^{t} \cL V (X_s , Y_s ) ds + \sigma \int_0^{t} \partial_x V (X_s , Y_s )\, d W_s + \int_0^{t} \big\langle \nabla_y V (X_s , Y_s ) , d I^c_s \big\rangle ,
\end{align}
we show
\begin{equation}
S^V_{t\wedge \tau} = R^V_{t\wedge \tau}, \qquad t\geq 0,
\label{eq:ito_bis}
\end{equation}
with $\tau:= \tau_{2,\text{sat}}\wedge \tau_{\text{sat}}$, where $\tau_{2,\text{sat}}$, $\tau_{\text{sat}}$ are as defined in \eqref{eq:tausat}-\eqref{eq:tausat_bis}.}

\vspace{2pt}

\emph{Step 1 (localization).} We show that it is enough to prove
\begin{equation}
S^V_{t\wedge \tau_{\delta}} = R^V_{t\wedge \tau_{\delta}} , \qquad t\geq 0, 
\label{eq:ito_ter}
\end{equation}
for any $0<\delta<<1$, where 
\begin{equation}\label{eq:tau_del}
\tau_{\delta}:=\begin{cases} 
\inf J_{\delta}& \text{if } J_{\delta} \neq \emptyset  \\
+\infty & \text{otherwise}
\end{cases}
\end{equation}
with
\begin{equation}
J_{\delta}:=  \{ t\geq 0 : Y_t \notin D_{\delta}\} , \qquad D_{\delta} : = \{y\in D : y_1+y_2 \leq \theta - \delta, y_2\leq \theta/2 - \delta \} .
\end{equation}
As $Y^1$ and $Y^2$ are cadlag and non-decreasing, $\tau_{\delta} \to \tau$, $\Pb$-almost surely, as $\delta \to 0^+$. Also, notice that $R^V$ is continuous. Finally, the fact that $Y$ is cadlag implies that $S^V_t = S^V_{t^-}$ for any $t> 0$. Therefore, we have
\begin{equation}
R^V_{t\wedge \tau_{\delta}} \longrightarrow R^V_{t\wedge \tau} , \  S^V_{t\wedge \tau_{\delta}} \longrightarrow S^V_{t\wedge \tau} \quad \Pb\text{-almost surely},\quad \text{as } \delta\to 0^+,
\end{equation}
for any $t\geq 0$. 

\vspace{2pt}

\emph{Step 2 (regularization in $x$).} Fix $0<\delta<<1$ and prove \eqref{eq:ito_ter} under the additional assumption $\nabla_y V \in C (\R\times D_{\eps})$ for some $\eps<\delta$. Denoting by $\phi$ a standard mollifier on $\R$, we set a smooth (in $x$) sequence $V^n$ defined as
\begin{equation}\label{E1}
{{V}^n(x,y):=\int_{\R} n\phi\big(n(x-\xi)\big) V(\xi,y) d\xi
, \qquad (x,y)\in \R\times D_{\eps}.} 
\end{equation}
Now \eqref{eq:ito_ter} holds for $V^n$ by standard It\^o formula. We now want to pass to the limit as $n\to \infty$. Recalling assumption (i) in Definition \ref{def:regular_sol_VP}, first observe that $V^n, \partial_x V^n$ and $\nabla_y V^n$ converge to $V, \partial_x V$ and $\nabla_y V$, respectively, uniformly on compacts of $ \R\times {D_{\eps}}$. 
Therefore, we obtain 
\begin{equation}
S^{V^n}_{t\wedge \tau_{\delta}} - R^{V^n}_{t\wedge \tau_{\delta}} - \frac12 \sigma^2 \int_{0}^{t\wedge \tau_{\delta}} \partial_{xx} V^n (X_s , Y_s ) ds \longrightarrow S^{V}_{t\wedge \tau_{\delta}} - R^{V}_{t\wedge \tau_{\delta}} -  \frac12 \sigma^2  \int_{0}^{t\wedge \tau_{\delta}} \partial_{xx} V (X_s , Y_s ) ds, \qquad \text{as } n\to+\infty,
\end{equation}
$\Pb$-almost surely.

To conclude the proof of \eqref{eq:ito_ter} we then have to show that
\begin{equation}\label{eq:conv_second_der}
\int_{0}^{t\wedge \tau_{\delta}} \partial_{xx} V^n (X_s , Y_s ) ds \longrightarrow \int_{0}^{t\wedge \tau_{\delta}} \partial_{xx} V (X_s , Y_s ) ds, \qquad \text{as}\quad n\to \infty,
\end{equation}
$\Pb$-almost surely. Note that, by assumption, $\partial_{xx}V$ is continuous on $\overline{W_F}
$, and thus we have
\begin{equation}
\partial_{xx} V^n (x ,y ) \longrightarrow \partial_{xx} V (x ,y ),  \qquad \text{as}\quad n\to \infty,
\end{equation}
for any $(x,y)\in (\R\times D_{\eps}) \setminus \partial W_F$. Therefore, by Lemma \ref{lem:boundary} together with Fubini's theorem, we have 
\begin{equation}
{\bf 1}_{s\leq \tau_{\delta}} \partial_{xx} V^n (X_s , Y_s ) \longrightarrow {\bf 1}_{s\leq \tau_{\delta}} \partial_{xx} V (X_s , Y_s ) \qquad \text{for almost every }s\in[0,t], 
\end{equation}
$\Pb$-almost surely. Furthermore, by assumption (i) in Definition \ref{def:regular_sol_VP}, for any $R>0$ there exists a constant $L_R$ such that
\begin{equation}
|\partial_{x} V^n (x , y ) - \partial_{x} V^n (x' , y ) | \leq \int_{\R} n\phi (n\, \xi ) | V(x-\xi,y) - V(x' - \xi,y)| d\xi \leq L_{R} |x - x'|, \qquad x\in[-R,R],\ y\in {D}_{\eps},
\end{equation}
which yields 
\begin{equation}
|\partial_{xx} V^n (x , y )| \leq L_R, \qquad x\in[-R,R],\ y\in {D}_{\eps}.
\end{equation}
Therefore, \eqref{eq:conv_second_der} follows by Lebesgue's dominated convergence theorem, $\Pb$-almost surely.

\vspace{2pt}

\emph{Step 3 (regularization in $y$).} We fix $0<\delta<<1$ and prove \eqref{eq:ito_ter} by
regularizing w.r.t $y$. For any $0<\eps<\delta/2$ we set
\begin{equation}
{{V}^{\eps}(x,y)=\frac{1}{\eps^2}\int_{y_1}^{y_1+\eps}\int_{y_2}^{y_2+\eps}{V}(x,z)dz_1dz_2, \quad (x,y)\in \R\times D_{2\eps}.}
\end{equation}
As $V$ is continuous, we have $\nabla_{y}{V}^{\eps} \in C (\R\times  D_{2\eps})$ and
\begin{align}
\p_{y_1}{ V}^{\eps}(x,y)&=\frac{1}{\eps^2}\int_{y_2}^{y_2+\eps}\big({ V}(x,y_1+\eps,z_2)-{ V}(x,y_1,z_2)\big)dz_2, \label{E3}\\
\p_{y_2}{ V}^{\eps}(x,y)&=\frac{1}{\eps^2}\int_{y_1}^{y_1+\eps}\big({ V}(x,z_1,y_2+\eps)-{ V}(x,z_1,y_2)\big)dz_1.\label{E4}
\end{align}
Also, owing to assumption (i) in Definition \ref{def:regular_sol_VP}, we have $\partial_{xx}V^{\eps}\in C(\overline{W_F} \cap (\R\times D_{2\eps}))$ and $\partial_{x}V^{\eps} $ satisfying the Lipschitz condition \eqref{eq:Lip_x} for any $y\in D_{2\eps}$. In particular,
\begin{equation}
\p^k_{x}{V}^{\eps}(x,y)=\frac{1}{\eps^2}\int_{y_1}^{y_1+\eps}\int_{y_2}^{y_2+\eps}\p^k_{x}{V}(x,z)dz_1dz_2, \qquad k=1,2.\label{E2}
\end{equation}
Therefore, Step 2. yields \eqref{eq:ito_ter} for $V^{\eps}$. We now need to pass to the limit as $\eps\to 0^+$.

Owing again to assumption (i) in Definition \ref{def:regular_sol_VP}, we have that $V^{\eps}, \partial_x V^{\eps}$ and $\partial_{xx} V^{\eps}$ converge to $V, \partial_x V$ and $\partial_{xx} V$, respectively, uniformly on compacts of $\overline{W_F}\cap (\R\times D_{2\eps})$. As $(X,Y)$ lives on $\overline{W_F}$, this yields
\begin{equation}
S^{V^{\eps}}_{t\wedge \tau_{\delta}} - R^{V^{\eps}}_{t\wedge \tau_{\delta}} - \int_0^{{t\wedge \tau_{\delta}}} \big\langle \nabla_y V^{\eps} (X_s , Y_s ) , d I^c_s \big\rangle \longrightarrow S^{V}_{t\wedge \tau_{\delta}} - R^{V}_{t\wedge \tau_{\delta}} -\int_0^{{t\wedge \tau_{\delta}}} \big\langle \nabla_y V (X_s , Y_s ) , d I^c_s \big\rangle, \qquad \text{as}\quad \eps\to 0^+,
\end{equation}
$\Pb$-almost surely. To conclude the proof of \eqref{eq:ito_ter} we then have to show that
\begin{equation}\label{eq:conv_second_der_bis}
 \int_0^{{t\wedge \tau_{\delta}}} \partial_{y_i} V^{\eps} (X_s , Y_s )  d I^{i,c}_s  \longrightarrow  \int_0^{{t\wedge \tau_{\delta}}} \partial_{y_i} V (X_s , Y_s )  d I^{i,c}_s  , \qquad \text{as}\quad \eps\to 0^+.
\end{equation}
$\Pb$-almost surely, for $i=1,2$. Let us first consider the term involving $\partial_{y_1}V^{\eps}$. Notice that, for any $(x,y)\in \overline{W_F} \cap (\R\times D_{\delta})$, we have $(x,\eta)\in \overline{W_F}$ for all $\eta\in (y_1,y_1+\eps) \times (y_2,y_2+\eps)$. Therefore, by applying mean-value theorem to \eqref{E3} and owing to $\partial_{y_1}V\in C(\overline{W_F})$ we obtain
\begin{equation}
\partial_{y_1}V^{\eps}(x,y)\longrightarrow \partial_{y_1}V(x,y), \qquad \text{as } \eps\to 0^+,
\end{equation}
for any $(x,y)\in \overline{W_F} \cap (\R\times D_{\delta})$, which implies
\begin{equation}
{\bf 1}_{s\leq \tau_{\delta}} \partial_{y_1}V^{\eps}(X_s,Y_s) \longrightarrow  {\bf 1}_{s\leq \tau_{\delta}}  \partial_{y_1}V (X_s,Y_s) , \qquad \text{as } \eps\to 0^+.
\end{equation}
As $\partial_{y_1}V^{\eps}$ is bounded on compacts of $\overline{W_F} \cap (\R\times D_{\delta})$, uniformly in $\eps$, and the process $(X,Y)$ lives on $\overline{W_F}$, Lebesgue dominated convergence theorem yields \eqref{eq:conv_second_der_bis} for $i=1$.
Finally, a nearly identical argument can be applied to show \eqref{eq:conv_second_der_bis} for $i=2$, owing to \eqref{eq:cond_b_i_optimal} with $j=2$ and the fact that $\p_{y_j} V\in C(U_r)$, with $
U_r := \{(x,y) \in \R^3 : (x+r,y)\in I_F\}$, for some $r>0$ (see condition (ii) in Definition \ref{def:regular_sol_VP}). This concludes the proof. 
\end{proof}

\begin{proof}[Proof of Theorem \ref{th:verific_th}]
\emph{Step 1.} For fixed $\{i,j\}$ permutation of $\{1,2\}$ and $(\bar{x},\bar{y})\in \R\times {D}$, we first prove that
\begin{equation}\label{eq:prof_ver_part1}
\cS_i(\bar{x},\bar{y}; I) \leq V_i(\bar{x},\bar{y}),
\end{equation}
for any $I\in \cI(\bar{y})$ such that \eqref{eq:cond_Iit} holds true. 
By symmetry, it is enough to consider $\{ i,j \} = \{1,2\}$. Furthermore, we ease notation by removing the upper indeces in $X^{\bar{x},\bar{y},I}$ and $Y^{\bar{y},I}$.

All the next equalities and inequalities are meant $\Pb$-almost surely. By Lemma \ref{lem:ito} (It\^o formula), together with stochastic integration-by-parts formula, we obtain
\begin{equation}
e^{- \rho  T } V_1(X_{ T }, Y_{ T }) - V_1(\bar{x}, \bar{y}) = H_1 + H_2 + H_3 + H_4 + H_5,
\end{equation}
where
\begin{align}
H_1&= \int_{0}^{ T } \big( e^{-\rho  t} \cL V_1 (X_t , Y_t ) -\rho e^{-\rho t } V_1(X_{t}, Y_{t})  \big) dt ,\\
H_2&= \int_{0}^{ T } e^{-\rho  t}\, \partial_{y_1} V_1 (X_t , Y_t )\, d I^{1,c}_t,&&
H_3= \int_{0}^{ T } e^{-\rho  t}\, \partial_{y_2} V_1 (X_t , Y_t )\, d I^{2,c}_t ,\\
H_4&= \sum_{0\leq t \leq  T } e^{-\rho t} \big( V_1 (X_t , Y_t ) - V_1 (X_t , Y_{t-} )  \big), &&
H_5=\sigma \int_{0}^{ T } e^{-\rho  t}\, \partial_x V_1 (X_t , Y_t )\, d W_t . 
\end{align}
By assumption, the control $I$ is $2$-optimal w.r.t. $F_2=F$ and $(\bar x, \bar y)$ in the sense of Definition \ref{def:i_optimal_control}. Therefore, by \eqref{eq:cond_a_i_optimal} and the fact that $(V_1,F_1)$ is a regular solution to $\text{VP}_{\{1,2\}}(F_2)$ (in particular by conditions (iii)-(iv) in Definition \ref{def:regular_sol_VP}), we have
\begin{equation}\label{eq:ineq_prof_verif_1}
H_1 \leq - \int_{0}^{ T } e^{-\rho  t} X_t  Y_t \,  dt, \qquad 
H_2 \leq c \int_{0}^{ T } e^{-\rho  t} \, d I^{1,c}_t.
\end{equation}
As for $H_3$, \eqref{eq:cond_b_i_optimal} yields
\begin{equation}
H_3= \int_{0}^{ T } e^{-\rho  t}\, \partial_{y_2} V_1 (X_t , Y_t ) {\bf 1}_{\{ (X_t , Y_t ) \in U_r \cap \overline{\Wb^2}  \}}  \, d I^{2,c}_t, 
\end{equation}
{for any $r>0$ suitably small, where $U_r$ is defined as in Definition \ref{def:regular_sol_VP}-(ii). As $\partial_{y_2} V_1\in C(U_r)$, and by \eqref{eq:regular_sol_ter}, this yields $H_3=0$.} 
We now study $H_4$: as $I^2$ is continuous we have 
\begin{equation}
H_4 = H_{4,1} + H_{4,2} + H_{4,3} ,
\end{equation}
where
\begin{align} 
H_{4,1} & = V_1 (\bar{x} , \bar{y} + I_0 ) - V_1 (\bar{x} , \bar{y}_{1} , \bar{y}_{2} + I^2_0 ),\\
H_{4,2} & = V_1 (\bar{x} , \bar{y}_1 , \bar{y}_2 + I^2_0 )   - V_1 (\bar{x} , \bar{y}), \\
H_{4,3} & = \sum_{0 < t \leq  T } e^{-\rho t} \big( V_1 (X^{\bar{x},\bar{y},I}_t , Y^{\bar{y},I,1}_t, Y^{\bar{y},I,2}_t ) - V_1 (X^{\bar{x},\bar{y},I}_t , Y^{\bar{y},I,1}_{t-} , Y^{\bar{y},I,2}_{t} )  \big)  .
\end{align}
We have
\begin{equation}\label{eq:ineq_prof_verif_2}
{H}_{4,1} = \int_{0}^{I^1_0} \partial_{y_1} V_1 (\bar{x} , \bar{y}_{1} + s , \bar{y}_{2} + I^2_0 ) ds \leq c\, I^1_0, 
\end{equation}
where the last inequality stems from \eqref{eq:regular_sol}-\eqref{eq:regular_sol_bis} as $(\bar{x} , \bar{y}_{1} + s , \bar{y}_{2} + I^2_0) \in \overline{\Wb^2}$ for any $s\in [0,I^1_0]$. The latter is true since $(\bar{x} , \bar{y}_{1}  , \bar{y}_{2} + I^2_0 ) \in \overline{\Wb^2}$ (by \eqref{eq:cond_c_i_optimal}) and since $F_2$ is increasing in $y_1$. An analogous argument yields
\begin{equation}\label{eq:sum_jumps}
{H}_{4,3} \leq c \sum_{0 < t \leq  T } e^{-\rho t}( I^1_t - I^1_{t-} ).
\end{equation}
On the other hand, \eqref{eq:cond_c_i_optimal} yields $(\bar{x} , \bar{y}_{1}  , \bar{y}_{2} + s) \in {\Ib^2}$ for any $s\in [0,I^2_0]$ if $I^2_0>0$, and thus \eqref{eq:regular_sol_ter} implies
\begin{equation}
{H}_{4,2} = \int_{0}^{I^2_0} \partial_{y_2} V_1 (\bar{x} , \bar{y}_{1}  , \bar{y}_{2} + s ) ds =  0 . 
\end{equation}
Finally, as $V_1$ satisfies the growth condition (v) in Definition \ref{def:regular_sol_VP} and $\int_0^t \Eb[X^2_t] dt$ is finite for any $t>0$,
 we obtain $\Eb[H_5] = 0$. 
Summing up, we proved
\begin{equation}
V_1(\bar{x}, \bar{y}) - \Eb\big[  e^{- T } V_1(X_{ T }, Y_{ T }) \big] \geq 
\Eb\Big[     \int_{0}^{ T } e^{-\rho  t} X_t\,  Y_t \,  dt -  c \int_{0}^{ T } e^{-\rho  t} \, d I^1_t   \Big].
\end{equation}
Owing once more to condition (v) of Definition \ref{def:regular_sol_VP} we obtain
\begin{equation}
V_1(\bar{x}, \bar{y}) + K \Eb\big[  e^{- T } (1 + X_{ T }) \big] \geq 
\Eb\Big[     \int_{0}^{ T } e^{-\rho  t} X_t\,  Y_t \,  dt -  c \int_{0}^{ T } e^{-\rho  t} \, d I^1_t   \Big].
\end{equation}
Taking the limit as $ T \to + \infty$ and applying the Lebesgue dominated convergence theorem proves \eqref{eq:prof_ver_part1}. We refer to \cite[Section 3 and Appendix B]{MR4305783} for the details of the limit argument.

\vspace{2pt}

\emph{Step 2.} We prove \eqref{eq:nash_cond_th}, which, together with \eqref{eq:prof_ver_part1}, proves that $\contr^{\star}$ is a Nash equilibrium. By symmetry, it is enough to prove it for $i=1$.

First note that, by assumption, the control $I^{\star}$ as in \eqref{eq:I_markov_star} is $j$-optimal with respect to $F_j$ and $(\bar x,\bar y)$ for any $(\bar x,\bar y)\in\R \times D$ and $j=1,2$. Therefore, by property \eqref{eq:cond_a_i_optimal} of Definition \ref{def:i_optimal_control}, together with Lemma \ref{lem:boundary}, we have
\begin{equation}
\Pb\Big( (X^{\bar x,\bar y,I}_t,Y^{\bar y,I}_t) \notin \Wb^1 \cap \Wb^2 , Y^{\bar y,I}_t \in D \Big) = 0 ,
\end{equation}
and since $(V_1,F_1)$ is a regular solution to $\text{VP}_{\{1,2\}}(F_2)$ (in particular by conditions (iii)-(iv) in Definition \ref{def:regular_sol_VP}), we obtain
\begin{equation}
\cL V_1 (X^{\bar{x},\bar{y},I^{\star}}_t , Y^{\bar{y},I^{\star}}_t ) -\rho  V_1(X^{\bar{x},\bar{y},I^{\star}}_{t}, Y^{\bar{y},I^{\star}}_{t}) = - X^{\bar{x},\bar{y},I^{\star}}_t\,  Y^{\bar{y},I^{\star},1}_t , \qquad t>0,
\end{equation}
$\Pb$-almost surely, and thus the first inequality in \eqref{eq:ineq_prof_verif_1} becomes an equality. Furthermore, as $(\overline{\Wb^1} \cap \overline{\Wb^2})\setminus \Ib^1 \subset \Wb^1$, by \eqref{eq:cond_b_i_optimal} we have
\begin{equation}\label{eq:diff_Y_zero}
(X^{\bar x,\bar y,I^{\star}}_t, Y^{\bar y,I^{\star}}_t) \in (\overline{\Wb^1} \cap \overline{\Wb^2})\setminus \Ib^1 \Rightarrow d Y^{\bar y,I^{\star},1}_{t+} = 0.
\end{equation}
Therefore, owing to ${\partial_{y_1} V_1} \in C(\overline{\Wb^2})$ and to the second relation in \eqref{eq:regular_sol_bis}, the second inequality in \eqref{eq:ineq_prof_verif_1} becomes an equality as well.
Furthermore, condition \eqref{eq:cond_c_i_optimal} reads 
\begin{equation}\label{eq:cond_c_i_optimal_bis}
I^{\star,2}_0 = \min\{ s\geq 0 : (\bar x,\bar y+s \, e_j)  \in \overline{\Wb^2} \},
\end{equation}
which implies
\begin{equation}
 \int_{0}^{I^1_0} \partial_{y_1} V_1 (\bar{x} , \bar{y}_{1} + s , \bar{y}_{2} + I^{\star,2}_0 ) ds = c\, I^{\star,1}_0.
\end{equation}
Therefore, 
the inequality in \eqref{eq:ineq_prof_verif_2} becomes an equality. Finally, the sum in \eqref{eq:sum_jumps} is null as $I^{\star}$ is continuous. Then we obtain
\begin{equation}
V_1(\bar{x}, \bar{y}) - \Eb\big[  e^{- T } V_1(X^{\bar{x},\bar{y},I^{\star}}_{ T }, Y^{\bar{y},I^{\star}}_{ T }) \big] = 
\Eb\Big[     \int_{0}^{ T } e^{-\rho  t} X^{\bar{x},\bar{y},I^{\star}}_t\,  Y^{\bar{y},I^{\star},1}_t \,  dt -  c \int_{0}^{ T } e^{-\rho  t} \, d I^{\star,1}_t   \Big].
\end{equation}
Passing to the limit as $ T \to\infty$ yields \eqref{eq:nash_cond_th} for $i=1$ and concludes the proof.
\end{proof}

\subsection{Construction of equilibrium strategies}\label{sec:construc}

Let $(V,F)$ be a fixed equilibrium solution to VP. We construct an associated Nash equilibrium $\contr^{\star}$, namely an admissible strategy such that \eqref{eq:nash_cond_th}-\eqref{eq:I_markov_star} holds true. {We recall Notation \ref{not:barriers} and Notation \ref{notation_V_F}, which will be employed through this section. 

In analogy with the one-step game (cfr. Section \ref{sec:free_bound_static}), we set ${F}^{-1}:[F(0,0),\infty)\to {[0,\theta/2]}$ as
\begin{equation}\label{eq:invF_dyn}
{F}^{-1}(x):= \begin{cases}
g^{-1}(x) &\quad \text{if}\quad x\in [{F}(0,0),{F}(\theta/2,\theta/2)],\\ 
\frac{\theta}{2} & \quad \text{if}\quad x> F(\theta/2,\theta/2 ),
\end{cases}
\end{equation}
where $[0,\theta/2]\ni \eta \mapsto g(\eta) := F(\eta,\eta) = F_1(\eta,\eta) = F_2(\eta,\eta)$. Note that $g$ is increasing as $F$ is increasing in both variables, thus $F^{-1}$ is well defined. 
{For any $\xi\in [0,\theta]$, also set
${F}_{\xi}^{-1}:[F(\xi,0),\infty)\to [0,(\theta-\xi)\wedge \xi]$} as
\begin{equation}\label{eq:invFi_dyn}
{F}_{\xi}^{-1}(x):= \begin{cases}
g_{\xi}^{-1}(x) &\quad \text{if}\quad x\in [{{F}(\xi,0),{F}({\xi,(\theta-\xi)\wedge \xi})}],\\ 
(\theta-\xi) \wedge \xi  & \quad \text{if}\quad x> {{F}({\xi,(\theta-\xi)\wedge \xi})}.
\end{cases}
\end{equation}
where $[0,(\theta-\xi) \wedge \xi]\ni \eta \mapsto g_{\xi}(\eta) := F(\xi,\eta) = F_2(\xi,\eta) = F_1(\eta,\xi)$. Note that $g_{\xi}$ is increasing as $F(\xi,\cdot)$ is increasing, thus $F_{\xi}^{-1}$ is well defined. 
\begin{remark}
Note that, given $(x,y)\in \R\times D$, and $I^{\star}\in \cI(y)$ an admissible control that is {$j$-optimal with respect to $F_j$ and $(x,y)$} for both $j=1$ and $j=2$, it must hold that $I^{\star}_0$ corresponds to the equilibrium strategy for the one-step game in \eqref{eq:eq_strat_one_s}, which we recall here for the reader's convenience:
\begin{equation}\label{eq:rep_I0_opt}
I^{\star}_0=\begin{cases}
(0,0), & (x,y)\in \mathbb{W}^1\cap \mathbb{W}^1,\\
(0, {F}^{-1}_{y_1}(x)-y_2), & (x,y)\in \mathbb{W}^1\cap \mathbb{I}^2, \\
({F}^{-1}_{y_2}(x)-y_1 , 0), & (x,y)\in \mathbb{I}^1\cap \mathbb{W}^2, \\
({F}^{-1}(x)-y_1, {F}^{-1}(x)-y_2), & (x,y)\in \mathbb{I}^1\cap \mathbb{I}^2.
\end{cases}
\end{equation}
\end{remark}

To define the candidate Nash strategy,} we set the function ${\bf F}^{-1}:\R\times[0,\theta]\to \R$ as
\begin{equation}\label{eq:F_bold}
{\bf F}^{-1}(x,\xi) := 
\begin{cases}
\theta/2 \wedge (\theta -\xi), & \qquad x >  F({\theta/2\vee \xi},\theta/2\wedge(\theta-\xi)),\\
 F^{-1}(x) ,      & \qquad  x\in ( F(\xi,\xi) , F(\theta/2,\theta/2)],\\
F_{\xi}^{-1}(x) ,      &\qquad x\in ( F(\xi,0) , F(\xi,\xi\wedge (\theta-\xi) ) ],\\
0  ,      & \qquad 
x <  F(\xi,0).
\end{cases}
\end{equation}
For any $(i,j)$ permutation of $\{1,2\}$, define
\begin{equation}\label{eq:I_markov_star_tris}
\contr^{\star,j}_t\big((x_s)_{s\leq t},(y_s)_{s\leq t}\big): = \Big( \sup_{s\leq t} {\bf F}^{-1}(x_s,y_{i,s-}) - y_{j,0-} \Big)^+ , \qquad t\geq 0, \ x\in C([0,t]), \ y\in \NCL{t}.
\end{equation}

\begin{theorem}\label{th_equilibrium}
Let $(V,F)$ be a fixed equilibrium solution to VP {such that the functions ${F}^{-1}$ and ${F}^{-1}_{\xi}$, $\xi\in[0,\theta]$, defined by \eqref{eq:invF_dyn}-\eqref{eq:invFi_dyn} are globally Lipschitz continuous.} 
The strategy $\contr^{\star}{=(\contr^{\star,1},\contr^{\star,2})}$ defined by \eqref{eq:I_markov_star_tris} is a Nash equilibrium. Furthermore, for any $(\bar x , \bar y)\in \R\times D$ we have 
\begin{equation}\label{eq:nash_cond_th_bis}
\cS_i(\bar x,\bar y; I^{\star}) = V_i(\bar x,\bar y) , \qquad   i = 1,2,
\end{equation}
where $I^{\star}\in\mathcal{I}(\bar y)$ is the only admissible control such that 
\begin{equation}\label{eq:I_markov_star_bis}
I^{\star}_t = \contr_t^{\star}\big((X^{\bar x,\bar y,I^{\star}}_s)_{s\leq t}, (Y^{\bar y,I^{\star}}_{s})_{s\leq t}\big), \qquad t\geq 0,
\end{equation}
with $(X^{\bar x,\bar y,I^{\star}}, Y^{\bar y,I^{\star}})$ being the unique solution to \eqref{model}-\eqref{installed_power} with $(x,y) = (\bar x , \bar y)$ and $I = I^{\star}$. Finally, $I^{\star}$ is continuous.
\end{theorem}
\begin{proof} By Lemma \ref{lem:admissibility} and Lemma \ref{lem:optimal} below, we can apply (verification) Theorem \ref{th:verific_th} and obtain the result. In particular, $I^{\star}$ is continuous because, by Lemma \ref{lem:optimal}, $I^{\star}$ is $j$-optimal w.r.t. $F_j$ and $(\bar{x},\bar{y})$, for both $j=1,2$, in the sense of Definition \ref{def:i_optimal_control}.
\end{proof}
\begin{remark}
By inspecting the proof of Lemma \ref{lem:admissibility} below, one can represent the control $I^{\star}$ similarly to \cite{MR4305783}, namely as the unique (continuous) solution to the system
\begin{equation}\label{eq:optimal_cont}
\begin{cases}
(X^{\bar x,\bar y,I^{\star}}_t,Y^{\bar y,I^{\star}}_t) \in \overline{\mathbb{W}^1} \cap \overline{\mathbb{W}^2}\\
dX^{\bar x, I^{\star}}_t=k\left(\mu-\b \sum_{i=1,2}Y_{t}^{\bar y,I^{\star},i}-X^{\bar x,I^{\star}}_t\right)dt+\sigma dW_s \\
dI^{\star,j}_t = {\bf 1}_{(X^{\bar x,\bar y,I^{\star}}_t,Y^{\bar y,I^{\star}}_t) \in \mathbb{I}_j} dI^{\star,j}_t, \qquad j=1,2 \\
\end{cases},
\end{equation}
together with $Y_{t}^{\bar y,I^{\star}} = \bar y + I^{\star}_t$ and with the initial conditions
\begin{equation}\label{eq:optimal_cont_bis}
\begin{cases}
X^{\bar x,\bar y,I^{\star}}_0 = \bar x\\
I^{\star,1}_0 = \big( {\bf F}^{-1}(\bar x,\bar y_2) - \bar y_1 \big)^+ \\ 
I^{\star,2}_0 = \big( {\bf F}^{-1}(\bar x,\bar y_1) - \bar y_2 \big)^+ 
\end{cases}.
\end{equation}
{Note that this representation for $I^{\star}$ is equivalent to the one in \eqref{eq:rep_I0_opt}.}
\end{remark}

We complete the section with the Lemmas that appear in the proof of Theorem \ref{th_equilibrium}.

\begin{lemma}\label{lem:optimal}
For any $\{ i, j \}$ permutation of $\{ 1,2 \}$
 and $(\bar{x},\bar{y})\in \R\times {D}$, every control process $I\in\mathcal{I}(\bar y)$ such that 
\begin{equation}\label{eq:cond_Iit_bis}
I^j_t := \Big( \sup_{s\leq t} {\bf F}^{-1}(X^{\bar x, \bar y, I}_s , Y^{\bar y, I, i}_{s-}) - \bar y_{j} \Big)^+ , \qquad t\geq 0, 
\end{equation}
is $j$-optimal w.r.t. $F_j$ and $(\bar{x},\bar{y})$ in the sense of Definition \ref{def:i_optimal_control}. 
\end{lemma}
\begin{proof}
By symmetry, it is enough to consider the case $\{ i, j \} = \{ 1,2 \}$. Furthermore, to ease the notation, we suppress the explicit dependence on $\bar x, \bar y, I$ in $X^{\bar x, \bar y, I}$ and $Y^{\bar y, I, i}$.

We first check the continuity of $I^2$. Note that the process $Y^1_{s-}$ is left-countinuous because $I\in \mathcal{I}(\bar y)$ by assumption. Therefore, by 
the continuity of $X$ {and ${\bf F}^{-1}$}, we have that $I^2_t$ is also left-continuous. Furthermore, we have 
\begin{align}
\Delta I^2_t := I^2_{t+} - I^2_{t} & = \Big( \sup_{s\leq t} {\bf F}^{-1}(X_s , Y^1_{s}) - \bar y_{j} \Big)^+  - \Big( \sup_{s\leq t} {\bf F}^{-1}(X_s , Y^1_{s-}) - \bar y_{j} \Big)^+  \\
&\leq {\bf F}^{-1}(X_t , Y^1_{t}) - {\bf F}^{-1}(X_t , Y^1_{t-}) \leq 0,
\end{align}
where the last inequality stems from the fact that $Y^1_t\geq Y^1_{t-}$ and that the function ${\bf F}(x,\cdot)$ is non-increasing (this stems from the fact that $F(y_1,y_2)$ is increasing in both variables). On the other hand, by \eqref{eq:cond_Iit_bis} one simply has $\Delta I^2_t\geq 0$, and thus $\Delta I^2_t = 0$.

We now prove \eqref{eq:cond_a_i_optimal} by contradiction. Assume there exists $t>0$ such that, with positive probability, 
\begin{equation}
{(X_t,Y_t) \not\in \overline{W_F} \text{ and } Y_t \in D.}
\end{equation}
Therefore, we have 
\begin{equation}
{\bf F}^{-1}(X_t , Y^{1}_{t}) > Y^2_{t}.
\end{equation}
As $F(y_1,y_2)$ is increasing in both variables, we also have 
\begin{equation}
{\bf F}^{-1}(X_t , Y^{1}_{t-}) \geq  {\bf F}^{-1}(X_t , Y^{1}_{t}) .
\end{equation}
By \eqref{eq:cond_Iit_bis} and by $\bar y_2\leq Y^2_{t}$, we obtain
\begin{equation}
Y^2_{t} \geq \bar y_{2} +  \big(  {\bf F}^{-1}(X_t , Y^1_{t-}) - \bar y_{2} \big)^+ \geq  \bar y_{2} + \big(  {\bf F}^{-1}(X_t , Y^1_{t}) - \bar y_{2} \big)^+ = {\bf F}^{-1}(X_t , Y^1_{t}) > Y^2_{t},
\end{equation}
which is a contradiction. To show \eqref{eq:cond_b_i_optimal}, we need to prove that, if $(X_t, Y_t) \in W_F$, then there exits a (random) $\eps>0$ such that 
\begin{equation}\label{eq:dYtzero}
Y^2_s=Y^2_t, \qquad t\leq s< t+\eps.
\end{equation}
If $Y^2_t\geq \theta/2$ the claim is obvious by definition of ${\bf F}^{-1}$. Assume then $Y^2_t< \theta/2$. 
First note that $ (X_t,Y_t) \in {W_F} $ implies that $ (X_s,Y_t) \in {W_F} $ for $s> t$ sufficiently close to $t$, and thus
\begin{equation}
{\bf F}^{-1}(X_s , Y^{1}_{t}) < Y^2_{t}.
\end{equation}
Furthermore, as $F(y_1,y_2)$ is increasing in both variables, we also have 
\begin{equation}
{\bf F}^{-1}(X_s , Y^{1}_{s}) \leq {\bf F}^{-1}(X_s , Y^{1}_{t}).
\end{equation}
Therefore, for $s>t$ sufficiently close to $t$ we have ${\bf F}^{-1}(X_s , Y^{1}_{s}) < Y^2_{t}$, which (by \eqref{eq:cond_Iit_bis}) implies $Y^2_s = Y^2_t$.

Finally, \eqref{eq:cond_c_i_optimal} stems directly from 
\begin{equation}
I^{2}_0 = \big( {\bf F}^{-1}(\bar x,\bar y_1) - \bar y_2 \big)^+,
\end{equation}
which in turn is obvious by \eqref{eq:cond_Iit_bis}.
\end{proof}
\begin{lemma}\label{lem:admissibility}
Under the assumptions of Theorem \ref{th_equilibrium}, we have $\contr^{\star}\in\mathcal{A}$ in the sense of Definition \ref{def:markov_strategy}. 
\end{lemma}
\begin{proof}
Let $(\bar x , \bar y)\in \R\times D$ be fixed throughout the proof. We have to show that there exists a unique pair $(X^{\bar x,I^{\star}}_t,I^{\star}_t)_{t\geq 0}$, with $X^{\bar x,I^{\star}}$ continuous and $I^{\star}$ c\`adl\`ag, satisfying
\eqref{eq:I_markov_star_bis} and
\begin{equation}\label{eq:system_XY_proof}
\begin{cases}
X^{\bar x,I^{\star}}_t = \bar x + k \int_0^t \big(\mu-\b \sum_{i=1,2}Y_{s}^{\bar y,i}-X^{\bar x,I^{\star}}_s\big)ds+\sigma W_t\\
Y^{\bar y,I^{\star}}_t = \bar y + I^{\star}_t
\end{cases}, \qquad t \geq0.
\end{equation}
By symmetry, it is enough to consider $\bar y_2\leq \bar y_1$. {Furthermore, we only prove the case $\bar y_1 <\theta/2$ for sake of brevity, the case $\bar y_1 \geq \theta/2$ being simpler.} {Finally, to ease notation, we will neglect the explicit dependence on $\bar x$ and $\bar y$ 
in $X^{\bar x,I^{\star}}$ and $Y^{\bar y,I^{\star}}$.} 
\vspace{2pt}

\emph{Step 1: pathwise uniqueness.} 
Let $(X^{I^{\star}}_t,I^{\star}_t)_{t\geq 0}$ be a solution to the system \eqref{eq:I_markov_star_bis}-\eqref{eq:system_XY_proof}. 
By Lemma \ref{lem:optimal}, we have that $I^{\star}$ is $j$-optimal w.r.t. $F_j$ and $(\bar{x},\bar{y})$, for both $j=1,2$, in the sense of Definition \ref{def:i_optimal_control}. 
This, recalling that we are assuming $\bar y_2 \leq \bar y_1 < \theta/2$, implies that 
\begin{align}
&I^{\star,1}_t = 0,& &  \hspace{-100pt}  0\leq t<\tau, \label{eq:I1tau}   \\
&Y^{I^{\star},1}_t  = Y^{I^{\star},2}_t ,& & \hspace{-100pt}   t\geq \tau,\label{eq:Ytau}
\end{align}
where $\tau$ is the stopping time defined as
\begin{equation}
 \tau := 
\begin{cases}
\inf  H & \text{if }  H\neq\emptyset\\
+\infty & \text{if }  H= \emptyset
\end{cases}, \qquad\quad \begin{aligned}
 H &:= 
  \{  t\geq 0:  X_{t} \geq F(\bar y_1 , \bar y_1)  \}.
\end{aligned}
\end{equation}

Denoting now by $(X^{\tilde I^{\star}}_t, \tilde I^{\star}_t)_{t\geq 0}$ a second solution to the system \eqref{eq:I_markov_star_bis}-\eqref{eq:system_XY_proof}, by  \eqref{eq:I1tau} we obtain
\begin{equation}
I^{\star,2}_{t\wedge\tau} - \tilde I^{\star,2}_{t\wedge\tau}    = \Big( \sup_{s\leq t\wedge\tau} {\bf F}^{-1}\big(X^{I^{\star}}_s , \bar y_1\big) - \bar y_{2} \Big)^+ - \Big( \sup_{s\leq t\wedge\tau} {\bf F}^{-1}\big( X^{ \tilde I^{\star}}_s , \bar y_1\big) - \bar y_{2} \Big)^+ ,
\end{equation}
which yields
\begin{align}
\big| I^{\star,2}_{t\wedge\tau} - \tilde I^{\star,2}_{t\wedge\tau}  \big| & \leq \Big| \sup_{s\leq t\wedge\tau} {\bf F}^{-1}\big(X^{I^{\star}}_s , \bar y_1 \big)  - \sup_{s\leq t\wedge\tau} {\bf F}^{-1}\big( X^{\tilde I^{\star}}_s , \bar y_1\big) \Big| \\
& \leq \sup_{s\leq t\wedge\tau} \big|  {\bf F}^{-1}\big(X^{I^{\star}}_s , \bar y_1 \big) -  {\bf F}^{-1}\big( X^{\tilde I^{\star}}_s , \bar y_1 \big)  \big| 
\intertext{(by 
{the Lipschitz regularity of $F^{-1}_{\bar y_1}$, and thus of ${\bf F}^{-1}(\cdot,\bar y_1)$})}
&\leq \kappa \sup_{s\leq t\wedge\tau}  \big| X^{I^{\star}}_s - X^{\tilde I^{\star}}_s    \big| 
\intertext{}
&\leq \kappa   \int_{0}^{t} \Big( \big| X^{I^{\star}}_{s\wedge\tau} - X^{\tilde I^{\star}}_{s\wedge\tau}    \big|  +\big| I^{\star,2}_{s\wedge\tau} - \tilde I^{\star,2}_{s\wedge\tau}  \big| \Big) ds.
\end{align}
Gronwall's inequlity then yields 
\begin{equation}\label{eq:uniqueness_proof}
\big(X^{I^{\star}}_t, I^{\star,2}_{t} \big) = \big(X^{\tilde I^{\star}}_t, \tilde I^{\star,2}_{t} \big) , \qquad t\in[0,\tau].
\end{equation}
On the event $\{ \tau < \infty \}$, we also have
\begin{align}
I^{\star,2}_{\tau + t} - \tilde I^{\star,2}_{\tau + t}  &  = \Big( \sup_{s\leq \tau + t} {\bf F}^{-1}\big(X^{I^{\star}}_s , Y^{I^{\star}, 1}_s\big) - \bar y_{2} \Big)^+ - \Big( \sup_{s\leq \tau + t} {\bf F}^{-1}\big( X^{\tilde I^{\star}}_s ,  Y^{\tilde I^{\star}, 1}_s  \big) - \bar y_{2} \Big)^+ 
\intertext{(by \eqref{eq:Ytau})}
&  =  \sup_{\tau \leq s\leq \tau + t} { F}^{-1}\big(X^{I^{\star}}_s \big)  -  \sup_{ \tau \leq s\leq \tau + t} { F}^{-1}\big( X^{\tilde I^{\star}}_s   \big),
\end{align}
which, by 
the Lipschitz regularity of ${ F}^{-1}$, yields
\begin{align}
\big| I^{\star,2}_{\tau+t} - \tilde I^{\star,2}_{\tau+ t}  \big| & \leq \kappa \sup_{\tau \leq s\leq \tau + t }  \big| X^{I^{\star}}_s - X^{\tilde I^{\star}}_s    \big| 
\intertext{(by \eqref{eq:uniqueness_proof})}
&\leq \kappa   \int_{0}^{t} \Big( \big| X^{I^{\star}}_{\tau + s} - X^{\tilde I^{\star}}_{\tau + s}    \big| +\big| I^{\star,1}_{\tau + s} - \tilde I^{\star,1}_{\tau + s}  \big| +\big| I^{\star,2}_{\tau + s} - \tilde I^{\star,2}_{\tau + s}  \big| \Big) ds. \label{eq:uniqueness_proof_2}
\end{align}
By same argument, estimate \eqref{eq:uniqueness_proof_2} can be obtained for $\big| I^{\star,1}_{\tau+t} - \tilde I^{\star,1}_{\tau+ t}  \big|$. Therefore, Gronwall's inequality yields
\begin{equation}
\big(X^{I^{\star}}_t, I^{\star,1}_{t}, I^{\star,2}_{t} \big) = \big(X^{\tilde I^{\star}}_t, \tilde I^{\star,1}_{t} , \tilde I^{\star,2}_{t} \big) , \qquad t\in(\tau, \infty),
\end{equation}
which concludes the proof of pathwise uniqueness.

\vspace{2pt}

\emph{Step 2: weak existence.} 
Recalling that $W$ is a fixed Brownian motion on a given filtered probability space $(\Omega, \mathcal{F}, ( \mathcal{F}_t)_{t\geq 0},\Pb)$, we show that there exist a Brownian motion $\tilde W$ on the same probability space, and a pair $(X_t^{I^{\star}},I^{\star}_t)_{t\geq 0}$ satisfying \eqref{eq:I_markov_star_bis}-\eqref{eq:system_XY_proof} with $W$ replaced by $\tilde W$. 

Let $\Delta \in\R^2$ be defined as
\begin{equation}
\Delta_1 :=  \big( {\bf F}^{-1}(\bar x,\bar y_2) - \bar y_1 \big)^+ ,\qquad \Delta_2 :=  \big( {\bf F}^{-1}(\bar x,\bar y_1) - \bar y_2 \big)^+.
\end{equation}
As we are considering $\bar y_2 \leq \bar y_1$, we have
\begin{equation}
0\leq \Delta_1 \leq \Delta_2.
\end{equation}
Let $(X_t)_{t\geq 0}$ be the unique solution to the SDE
\begin{equation}
dX_t = k  \big(\mu - X_t \big) dt + \sigma\, d W_t, \qquad X_0 = \bar x,
\end{equation} 
and set
\begin{equation}
\tilde I_t :=
\Big( \sup\limits_{0 \leq s \leq t} {\bf F}^{-1}( X_s , \bar y_1) - (\bar y_2 + \Delta_2 ) \Big)^+ , \qquad t\geq 0 .
\end{equation} 
Note that $\tilde I$ is continuous (
as $ {\bf F}^{-1}$ is continuous by assumption), non-decreasing and $\tilde I_0 = 0$. 
Set now the stopping time
\begin{equation}
 \tau := 
\begin{cases}
\inf  H & \text{if }  H\neq\emptyset\\
+\infty & \text{if }  H= \emptyset
\end{cases}, \qquad\quad \begin{aligned}
 H &:=  \{  t\geq 0: \bar y_2 + \Delta_2 + \tilde I_t \geq \bar y_1  \}\\
 &\ =  \{  t\geq 0:  X_{t} \geq F(\bar y_1 , \bar y_1)  \}
\end{aligned}.
\end{equation}
Note that
\begin{equation}\label{eq:prop_stop_bar_pen}
\Delta_1 = 0 \Longleftarrow x< F(\bar y_1, \bar y_1) \Longleftrightarrow \bar y_2 + \Delta_2 < \bar y_1  \Longleftrightarrow  \tau > 0   
\end{equation}
and
\begin{equation}\label{eq:prop_stop_bar_exa}
\bar y_1 + \Delta_1 = \bar y_2 + \Delta_2+\tilde I_{ \tau} .
\end{equation}
Define now
\begin{equation}
\hat I_{t}: = \sup\limits_{ \tau \leq s \leq t}  F^{-1}(X_s) - ( \bar y_1 + \Delta_1 ) ,
 \qquad   t\geq \tau.
\end{equation} 
Owing to 
the continuity of $F^{-1}$ (again, by assumption), we have that $\hat I$ is continuous, non-decreasing and, by \eqref{eq:prop_stop_bar_pen} together with the continuity of $ X$, we also have $\hat I_{ \tau} = 0$. 
Finally, for any $t\geq 0$ we can set
\begin{equation}
I^{\star,1}_t := \begin{cases}
\Delta_1 & \text{if }  0\leq t<\tau \\
\Delta_1 + \hat I_t & \text{if }  t \geq\tau \\
\end{cases}, \qquad 
I^{\star,2}_t:= \begin{cases}
\Delta_2+\tilde I_t & \text{if }  0 \leq t<\tau \\
\Delta_2+\tilde I_{ \tau}  +\hat I_t & \text{if }  t \geq \tau \\
\end{cases}.
\end{equation}
Note that $\hat I_{\tau} = 0$ implies that the process $I^{\star} = (I^{\star,1},I^{\star,2})$ defined above is continuous. Also, \eqref{eq:prop_stop_bar_exa} yields 
\begin{equation}
\bar y_1 +  I^{\star,1}_s = \bar y_2 +  I^{\star,2}_s, \qquad s\geq \tau.
\end{equation}
Therefore, a direct inspection shows that the pair $(X,I^{\star})$ satisfies \eqref{eq:I_markov_star_bis}-\eqref{eq:system_XY_proof} with $\beta = 0$. 

Observe now that, by construction, the process $ I^{\star}$ is bounded. Thus, by Girsanov's theorem (see \cite{MR2791231}, Theorem 10.5), there exists a probability measure $\tilde\Pb$ under which the process 
\begin{equation}
\tilde W_t = W_t + \frac{k \beta}{\sigma} \int_0^t \big( \bar y_1 + I^{\star,1}_s + \bar y_2 + I^{\star,2}_s\big)  d s, \qquad t\geq 0,
\end{equation}
is a Brownian motion. Therefore, the pair $(X,I^{\star})$ satisfies \eqref{eq:I_markov_star_bis}-\eqref{eq:system_XY_proof} with $W$ replaced by $\tilde W$. This concludes the proof of weak existence for the system \eqref{eq:I_markov_star_bis}-\eqref{eq:system_XY_proof}.
\end{proof}

\subsection{On the regularity of the equilibrium solution to the HJB system}\label{sec:regularity_V_discus}

In this section we present an informal discussion about the structure and the regularity of the value function $V$, namely the equilibrium solution to the variational problem VP in the sense of Definition \ref{def:equilibrium_sol}.


\medskip
Recalling Notation \ref{notation_V_F}, we first observe that, by conditions (iii) in Definition \ref{def:regular_sol_VP}, on the joint waiting region 
the following elliptic system should be satisfied:
\begin{equation}\label{eq:1a}
(\cL^y-\rho)V_i(x,y)+x y_i =0,  \qquad i=1,2, \quad (x,y)\in {\mathbb{W}^1\cap\mathbb{W}^2}.
\end{equation}
Equation \eqref{eq:1a} can be interpreted as a second order ODE parameterized by the vector $y$, for which $R_i$ in \eqref{eq:Ri} is a particular solution. Following the arguments in \cite{MR4305783}, one finds that $V$ ($=V_1$) necessarily takes the form
\begin{equation}\label{eq:eq_Vv1}
V(x,y) = v_1(x,y), \qquad \text{on } {\mathbb{W}^1\cap\mathbb{W}^2},
\end{equation}
with
\begin{equation}\label{eq:2a}
{v_1(x,y)=m_1(y)\psi(x+\beta \langle \mathbf{1},y \rangle)+R_i(x,y)}, \qquad \psi(z)=\frac{1}{\Gamma(\frac \rho k)}\int_0^{\infty}t^{\frac \rho k -1}e^{-\frac{t^2}{2}+\frac{z-\mu}{\sigma}t\sqrt{2k}}dt.
\end{equation}
Here, $\psi$ 
is a positive and strictly increasing fundamental solution for the homogeneous ODE $$(\cL^{(0,0)}-\rho)V(x,y)=0,$$
and $m_1(y)$ is a function of $y$ to be determined. Note that the boundary condition \eqref{eq:terminal_cond_bis} (cf. Definition \ref{def:regular_sol_VP}-(iv)) implies $m(0)=0$.
Similarly to the one-step case, $R_1$ represents the value of selling permanently $y_1$ units of energy, which is the initial level of installed power, but here we also add the product between $m_1$ and $\psi$, which may be interpreted as the value of the option to increase the installed power. 
\begin{remark}\label{rem:fund_sol}
The fundamental solution $\psi$ in \eqref{eq:2a} is smooth. All its derivatives are positive, strictly increasing. 
\end{remark}
Employing once more conditions (iii) in Definition \ref{def:regular_sol_VP}, in particular by the transport equations in \eqref{eq:regular_sol_bis}-\eqref{eq:regular_sol_ter}, and by the continuity of $V$, we can deduce the structure of $V$ in the remaining regions of the domain as follows:
\begin{align}
V(x,y)=\begin{cases}
v_1(x,y), & (x,y)\in \mathbb{W}^1\cap \mathbb{W}^2, \\
v_1\big(x,y_1,F^{-1}_{y_1}(x)\big), & (x,y)\in \mathbb{W}^1\cap \mathbb{I}^2 ,\\
v_1\big(x, F^{-1}_{y_2}(x),y_2\big)-c\big(F^{-1}_{y_2}(x)-y_1\big), & (x,y)\in \mathbb{I}^1\cap \mathbb{W}^2 , \\
v_1\big(x, F^{-1}(x),F^{-1}(x)\big)-c\big( F^{-1}(x)-y_1\big), & (x,y)\in \mathbb{I}^1\cap \mathbb{I}^2,
\end{cases}\label{eq:value}
\end{align}
where the functions $F^{-1}$ and $F^{-1}_{y_1}$ are as defined in \eqref{eq:invF_dyn} and \eqref{eq:invFi_dyn}, respectively.
Let us interpret the system above in relation to the optimal strategy in \eqref{eq:I_markov_star_tris}-\eqref{eq:I_markov_star_bis}.
Formally speaking, when the state $(X_t,Y_t)$ starts from the region $\mathbb{I}^1\cap \mathbb{W}^2$, Agent $1$ pushes the process to the boundary $\p F_1$ along the direction $(0,1,0)$, so as to increase his level of installed power by $F^{-1}_{y_2}(x)-y_1$. 
The associated payoff to this action is the difference between the continuation value starting from the new state 
$(x,F^{-1}_{y_2}(x),y_2)$ and the associated costs of installation $c\left( F^{-1}_{y_2}(x)-y_1\right)$. At the same time Agent $2$ enjoys a payoff given by his continuation value computed at the new state. 
In particular, if Agent $1$ has to restrict his action due to the capacity limit $\theta$, then the associated payoffs for Agents $1$ and $2$ are $R_1(x,\theta-y_2,y_2)-c(\theta-y_1-y_2)$ and $R_1(x,y_1,\theta-y_1)$, respectively. 
On the other hand, if the process $(X_t,Y_t)$ starts from $\mathbb{I}^1\cap \mathbb{I}^2$, then both agents install and push the process to $(x,F^{-1}(x), F^{-1}(x))$, so as to increase their level of installed power by $F^{-1}(x)-y_1$ and  $F^{-1}(x)-y_2$. In particular, if saturation occurs, then the associated payoffs for Agents $1$ and $2$ are $R_1(x,\theta/2,\theta/2)-c(\theta/2 - y_1)$ and $R_1(x,\theta/2,\theta/2)-c(\theta/2 - y_2)$, respectively. 

\medskip
With the representation \eqref{eq:4a} at hand, 
we can make an informed discussion about the regularity which is expected for the value function $V$. To start, notice that the functions $\psi$ and $R_1$ are smooth with respect to all variables, thus the equalities in \eqref{eq:regular_sol_bis}-\eqref{eq:regular_sol_ter} can be interpreted in the classical sense and the regularity of $V$ on its domain only depends on that of $m_1$ and of the free boundaries with respect to the variable $y$. 

\paragraph{Regularity in $y$.} By Definition \ref{def:regular_sol_VP}-(ii), we have $\p_{y_1}V_1 \in C (\overline{\mathbb{W}^2})$. This and \eqref{eq:eq_Vv1}-\eqref{eq:2a} yield $\partial_{y_1} m_1\in C(D)$. Also, by \eqref{eq:regular_sol_bis} $\p_{y_1}V_1$ is 
constant and equal to $c$ on $\mathbb{I}^1\cap \mathbb{W}^2$, and thus $\p_{y_1}V_1 \in C (\overline{\mathbb{W}^2})$ if and only if the following ($1$st order) smooth-fit condition holds} on the graph of the boundary $F_1$:
\begin{equation}\label{eq:4a}
c=\left(\p_{y_1}v_1\right)(F_1(y),y) , \qquad D\cap \{y_2\geq y_1\}.
\end{equation}
Furthermore, again by Definition \ref{def:regular_sol_VP}-(ii), we have $\p_{y_2} V\in C(U_r)$ for some $r>0$, with $U_r := \{(x,y) \in \R^3 : (x+r,y)\in \mathbb{I}^2\}$. This and \eqref{eq:eq_Vv1}-\eqref{eq:2a} yield $\partial_{y_2} m_1\in C(D\cap\{y_1<\theta/2\})$. Also, by \eqref{eq:regular_sol_ter}, $\p_{y_2}V_1 =0$ 
constantly on $\mathbb{I}^2$, and therefore the following ($1$st order) smooth-fit condition needs to hold on the graph of the boundary $F_2$:
\begin{equation}\label{eq:4a_quat}
0=\left(\p_{y_2}v_1\right)(F_2(y),y), \qquad D\cap \{y_2\leq y_1\}.
\end{equation}
Also, assuming both \eqref{eq:4a} and \eqref{eq:4a_quat}, one can check the continuity of $\p_{y_2} V$ on the boundary that separates $\mathbb{I}^1\cap \mathbb{W}^2$ from $\mathbb{I}^1\cap \mathbb{I}^2$, 
and thus that $\p_{y_2} V\in C(U_r)$.
\paragraph{Regularity in $x$.} By Definition \ref{def:regular_sol_VP}-(ii), we first have $\p_{x}V_1\in C(\R\times D)$. We show that the latter is ensured by \eqref{eq:4a}-\eqref{eq:4a_quat}. 
By 
\eqref{eq:value}, and recalling also \eqref{eq:Ri}, 
a direct computation shows that
\begin{equation}\label{eq:reg_x}
\p_x V_1(x,y)={\p_x v_1(x,y)}=m_1(y)\psi'(x+\beta\langle 1,y \rangle)+
\frac{y_1}{\rho+k}, \qquad (x,y)\in \mathbb{W}^1\cap \mathbb{W}^2.
\end{equation}
and
\begin{align}
\p_x V_1(x,y)&= \p_x F^{-1}_{y_2}(x)\left[({\p_{y_1}v_1})\big(x, F^{-1}_{y_2}(x),y_2\big)-c\right]+(\p_xv_1)\big(x, F^{-1}_{y_2}(x),y_2\big)\\ 
&=(\p_xv_1)\big(x, F^{-1}_{y_2(x)},y_2\big)\label{eq:reg_x211}\\
& = m_1\big( F^{-1}_{y_2}(x),y_2\big)\psi'\big(x+\beta\big\langle 1,(F^{-1}_{y_2}(x),y_2) \big\rangle\big)+
\frac{F^{-1}_{y_2}(x)}{\rho+k}, \qquad (x,y)\in \text{int} (\mathbb{I}^1\cap \mathbb{W}^2),
\label{eq:reg_x21}
\end{align}
where the second equality above stems from $ \p_x F^{-1}_{y_2}(x)=0$ in the regime $\theta>\theta/2$, $x\ge F_1(\theta-y_2,y_2)$, and from condition \eqref{eq:4a} in the other cases. Owing to the continuity of $m_1$, $\psi'$ and $F^{-1}_{y_2}(x)$ (in both variables $y_2,x$), \eqref{eq:reg_x}-\eqref{eq:reg_x21} implies 
that $\p_x V_1$ is continuous on $\mathbb{W}_2$. Similarly, employing also \eqref{eq:4a_quat}, one can show that $\p_x V_1$ is continuous on the boundary of $\mathbb{W}_2$ as well, and thus on $\R\times D$.

We now turn our attention to another regularity assumption in Definition \ref{def:regular_sol_VP}-(ii), namely $\p_{xx}V\in C(\overline{\mathbb{W}^2})$. 
We show that the latter is ensured as long as the following 
(2nd order) smooth-fit condition on the graph of the boundary $F_1$ holds:
\begin{equation}\label{eq:4aa}
0={\left(\p_{x y_1}v_1\right)(F_1(y),y)}, \qquad D\cap \{y_1\leq y_2\},
\end{equation}
with {$v_1$} as in \eqref{eq:2a}. 
By \eqref{eq:reg_x}, we obtain
\begin{equation}\label{eq:reg_x2}
\p_{xx} V_1(x,y)={\p_{xx} v_1(x,y)}=m_1(y)\psi''(x+\beta\langle 1,y \rangle), \qquad (x,y)\in \mathbb{W}^1\cap \mathbb{W}^2,
\end{equation}
and, by \eqref{eq:reg_x211},
\begin{align}
\p_{xx} V_1(x,y)&=\p_x F^{-1}_{y_2}(x)(\p_{x y_1}v_1)\big(x, F^{-1}_{y_2}(x),y_2\big)+(\p_{xx}v_1)\big(x, F^{-1}_{y_2}(x),y_2\big)\\
&=(\p_{xx}v_1)\big(x, F^{-1}_{y_2}(x),y_2\big) =m_1\big( F^{-1}_{y_2}(x),y_2\big)\psi''\big(x+\beta\big\langle 1,(F^{-1}_{y_2}(x),y_2) \big\rangle\big), \qquad (x,y)\in \text{int} (\mathbb{I}^1\cap \mathbb{W}^2),\label{eq:reg_x22}
\end{align}
where the second equality above stems from $ \p_x F^{-1}_{y_2}(x)=0$ in the regime $\theta>\theta/2$, $x\ge F_1(\theta-y_2,y_2)$, and from condition \eqref{eq:4aa} in the other cases.  Owing to the continuity of $m_1$, $\psi''$ and $F^{-1}_{y_2}(x)$ (in both variables $y_2,x$), \eqref{eq:reg_x2}-\eqref{eq:reg_x22} implies 
that $\p_{xx} V_1$ is continuous on $\overline{\mathbb{W}_2}$.

Finally, we need to address the Lipschitz condition \eqref{eq:Lip_x} for $\p_x V_1$. 
By \eqref{eq:reg_x2}-\eqref{eq:reg_x22}, together with Remark \ref{rem:fund_sol}, and observing that $F^{-1}_{y_2}(x)$ is bounded, we infer that $\p_{xx}V$ is bounded on $\mathbb{W}_2$. Likewise, recalling \eqref{eq:4a_quat} and the continuity of $\p_{y_2}m$ on {$D\cap\{y_2<\frac{\theta}{2}\}$}, one can show that $\p_{xx}V$ is also bounded on $\mathbb{I}_2$. The Lipschitz assumption \eqref{eq:Lip_x} is satisfied with $L_R=\max_{x\in [-R,R],y\in D}|\p_{xx}V(x,y)|$.

\paragraph{Sublinear growth in $x$.} 
With similar reasoning we obtain that $\partial_x V$ is bounded, and thus the growth assumption \eqref{eq:linear_growth_V} in Definition \ref{def:regular_sol_VP}-(v) is satisfied.

\medskip 

To sum up, we identified a system composed by three conditions, namely \eqref{eq:4a}-\eqref{eq:4a_quat}-\eqref{eq:4aa}, involving the values of $m_1$ and its derivatives $\partial_{y_1}m_1$, $\partial_{y_2}m_1$ on the free boundary, which guarantees that the regularity assumptions of Definition \ref{def:regular_sol_VP} are satisfied. By \eqref{eq:2a} and \eqref{eq:Ri}, this system can be written explicitly as
\begin{equation}
\begin{cases}
c=\p_{y_1}m_1(y)\psi (F_1(y)+\beta\langle {\bf 1}, y \rangle)+\beta m_1(y)\psi' (F_1(y)+\beta\langle {\bf 1}, y \rangle) +\frac{x\rho+\mu k - \b k \langle {\bf 1},y\rangle - \beta k y_1 }{\rho (\rho+k)}&  y\in D\cap \{y_1\leq y_2\},  \\
0=\p_{y_2}m_1(y)\psi (F_2(y)+\beta\langle {\bf 1}, y \rangle)+\beta m_1(y)\psi' (F_2(y)+\beta\langle {\bf 1}, y \rangle)-\frac{\b k y_1}{\rho (\rho+k)}&  y\in D\cap \{y_1\geq y_2\},  \\
0=\p_{y_1}m_1(y)\psi'(F_1(y)+\beta\langle {\bf 1}, y \rangle)+\beta m_1(y)\psi'' (F_1(y)+\beta\langle {\bf 1}, y \rangle)+(\rho+k)^{-1}, &  y\in D\cap \{y_1\leq y_2\}.  
\end{cases}\label{eq:system}
\end{equation}
System 
\eqref{eq:system} will serve as the starting point for inferring the unknown functions $m$ and $F_1$ and proving the inequalities in \eqref{eq:regular_sol}-\eqref{eq:regular_sol_bis}. 
This analysis requires a dedicated investigation and is the subject of ongoing research.

\bibliographystyle{acm}
\bibliography{bib}

\end{document}